\documentclass[preprint,3p,times]{elsarticle}
\usepackage{amssymb}
\usepackage{indentfirst}
\usepackage{array} 		%definisce nuove opzioni per lÕallineamento delle colonne di un ambiente tabular
\usepackage{newlfont}
\usepackage{microtype}
\usepackage{amscd}

\usepackage[fleqn,tbtags]{amsmath}
\usepackage{amssymb,mathrsfs} %per le formule matematiche fuori dal testo 
\usepackage{dsfont}     %\mathds{1}
\usepackage{mathtools}
\numberwithin{equation}{section}      %numera le equazioni con le sezioni

\usepackage{latexsym}
\usepackage{amsthm}
\usepackage{fancyhdr}

\theoremstyle{definition}                    %stile corsivo
\newtheorem{theorem}{Theorem}[section]      %definizione ambiente teorema
\newtheorem{corollary}[theorem]{Corollary}      %definizione ambiente corollario
\newtheorem{lemma}[theorem]{Lemma}            %definizione ambiente lemma
\theoremstyle{definition}               %stile roman
\newtheorem{definition}[theorem]{Definition}%  definizione ambiente definizione
\newtheorem{remark}[theorem]{Remark}

\def\R{\mathbb R}
\def\N{\mathbb N}

\def\E{\mathbb E}

\def\1{\mathds{1}}

\def\st#1#2{{\mathrel{\mathop{#2}\limits_{#1}}{}\!}}

\def\one{\mbox{$1\!\!\!\,\rule{0,1mm}{2,4mm}\,$}}

\def\Box{\rule{2mm}{2mm}}

\def\calf{{\mathcal{F}}}
\def\calh{\mathcal{H}}
\def\calo{\mathcal{O}}
\def\calf{\mathcal{F}}
\def\calk{{\mathcal{K}}}

\def\calh{{\mathcal{H}}}
\def\cals{{\mathcal{S}}}

\def\st#1#2{{\mathrel{\mathop{#2}\limits_{#1}}{}\!}}

\def\bit{\begin{itemize}}
\def\eit{\end{itemize}}
\def\vsp{\vspace*{1,5mm}\\ }
\def\n{\noindent }
\def\bk{\bigskip }
\def\mk{\medskip }

\def\oo{{\omega}}
\def\ooo{{\Omega}}
\def\<{\left<}
\def\>{\right>}
\def\({\left(}
\def\){\right)}
\def\ff{\forall }
\def\9{{\infty}}
\def\barr{\begin{array}}
\def\earr{\end{array}}
\def\ov{\overline}
\def\wt{\widetilde}
\def\vp{{\varepsilon}}

\def\pp{{\partial}}
\def\dd{\displaystyle}
\def\a{{\alpha}}
\def\b{{\beta}}
\def\vf{{\varphi}}
\def\lbb{{\lambda}}
\def\g{{\gamma}}
\def\pas{\mathbb{P}\mbox{-a.s.}}

\def\n{\noindent }
\def\na{{\nabla}}

\def\onr{\mbox{ on }}

\def\fwg{following}

 \def\1{^{-1}}

\biboptions{longnamesfirst,comma}

\journal{Journal de math\'ematiques pures et appliqu\'ees}

\begin{document}
%\linenumbers
\begin{frontmatter}

%% Title, authors and addresses

%% use the tnoteref command within \title for footnotes;
%% use the tnotetext command for the associated footnote;
%% use the fnref command within \author or \address for footnotes;
%% use the fntext command for the associated footnote;
%% use the corref command within \author for corresponding author footnotes;
%% use the cortext command for the associated footnote;
%% use the ead command for the email address,
%% and the form \ead[url] for the home page:
%%
%% \title{Title\tnoteref{label1}}
%% \tnotetext[label1]{}
%% \author{Name\corref{cor1}\fnref{label2}}
%% \ead{email address}
%% \ead[url]{home page}
%% \fntext[label2]{}
%% \cortext[cor1]{}
%% \address{Address\fnref{label3}}
%% \fntext[label3]{}

\title{Stochastic Porous Media Equations in $\R^d$
\\
Equations de milieux poreux dans $\R^d$}
\author[iasi]{Viorel BARBU}
\ead{vbarbu41@gmail.com}
\author[bielefeld]{Michael R\"OCKNER}
\ead{roeckner@math.uni-bielefeld.de}
\author[ensta]{Francesco RUSSO}
\ead{francesco.russo@ensta-paristech.fr}
\address[iasi]{Octav Mayer Institute of Mathematics (Romanian Academy), Blvd. Carol I no. 8, Ia\c si, 700506, Romania.}
\address[bielefeld]{Universit\"at Bielefeld, D-33501 Bielefeld, Germany.}
\address[ensta]{ENSTA ParisTech, Unit\'e de Math\'ematiques Appliqu\'ees,
 828, Bd. des Mar\'echaux, F-91120 Palaiseau}

\begin{abstract}
Existence and uniqueness of
solutions to the stochastic porous media equation
 $dX-\Delta\psi(X) dt=XdW$ in $\R^d$ are studied.
 Here, $W$ is a Wiener process, $\psi$ is a maximal monotone graph in $\R\times\R$ such that
 $\psi(r)\le C|r|^m$, $\ff r\in\R$.
In this general case, the dimension is restricted to $d\ge 3$, the main reason being the absence of a convenient multiplier result in the space $\calh=\{\varphi\in\mathcal{S}'(\R^d);\ |\xi|(\calf\varphi)(\xi)\in L^2(\R^d)\}$, for $d\le2$.
When $\psi$ is Lipschitz, the well-posedness, however, holds for all dimensions on the classical Sobolev space $H^{-1}(\R^d)$.
 If $\psi(r)r\ge\rho|r|^{m+1}$ and $m=\frac{d-2}{d+2}$, we prove the
finite time extinction with strictly positive probability.

\noindent
\\
\textbf{R\'esum\'e}\\

\noindent
Nous \'etudions existence et unicit\'e pour les solutions
d'une \'equation de milieux poreux
 $dX-\Delta\psi(X) dt=XdW$ dans $\R^d$.
 Ici $W$ est un processus de Wiener, $\psi$ 
est un graphe maximal monotone dans $\R\times\R$ tel que
 $\psi(r)\le C|r|^m$, $\ff r\in\R$.
Dans ce contexte g\'en\'eral, la dimension est restreinte \`a
 $d\ge 3$, essentiellement compte tenu de l'absence d'un
r\'esultat ad\'equat de multiplication dans l'espace
 $\calh=\{\varphi\in\mathcal{S}'(\R^d);\ |\xi|(\calf\varphi)(\xi)\in L^2(\R^d)\}$, pour $d\le2$.
Lorsque $\psi$ est Lipschitz, le probl\`eme est n\'eanmoins bien pos\'e 
pour toute dimension dans l'espace de  Sobolev classique $H^{-1}(\R^d)$.
 Si $\psi(r)r\ge\rho|r|^{m+1}$ et $m=\frac{d-2}{d+2}$, nous prouvons
 une propri\'et\'e d' extinction en temps fini avec probabilit\'e strictement
positive.

\begin{keyword} 
Stochastic \sep porous media \sep 
Wiener process \sep maximal monotone graph \sep distributions.

\MSC[2010] 76S05 \sep 60H15 \sep 35K55 \sep 35K67.
\end{keyword}
\end{abstract}

\end{frontmatter}

\section{Introduction}

Consider the stochastic porous media equation
\begin{equation}\label{e1.1}
\barr{l}
dX-\Delta\psi(X)dt=XdW\mbox{ in }(0,T)\times\R^d,\vsp
X(0)=x\mbox{ on }\R^d,\earr\end{equation}
where $\psi$ is a monotonically nondecreasing function on $\R$ (eventually multivalued) and $W(t)$ is a Wiener process of the form

\begin{equation}\label{e1.2}
W(t)=\sum^\9_{k=1}\mu_ke_k\b_k(t),\ t\ge0.\end{equation}
Here $\{\b_k\}^\9_{k=1}$ are   independent Brownian motions on a stochastic basis 
 $\{\ooo,\calf,\calf_t,\mathbb{P}\}$,  $\mu_k\in\R$ and $\{e_k\}^\9_{k=1}$ is an orthonormal basis in $H^{-1}(\R^d)$ or $\calh^{-1}$ (see \eqref{e2.2} below) to be made precise later on.

On bounded domains $\calo\subset\R^d$  with Dirichlet homogeneous boundary conditions,  equation \eqref{e1.1} was studied in \cite{BDR08}, \cite{BDR09}, \cite{BDPR09}, under general assumptions on $\psi:\R\to\ov\R$ (namely, maximal monotone multivalued graph with polynomial growth, or even more general growth conditions in \cite{BDR09}). It should be said, however, that there is a principial difference between  bounded and unbounded domains, mainly due to the multiplier problem in Sobolev spaces on $\R^d$. If $d\ge3$ and $\calo=\R^d$, existence and uniqueness of solutions to \eqref{e1.1} was proved in \cite{Ren} (see, also, \cite{15prim}) in a general setting which covers the case $\calo=\R^d$ (see Theo\-rems 3.9, Proposition 3.1 and Example 3.4 in \cite{Ren}). However, it should be said that in \cite{Ren} $\psi$ is assumed continuous, such that $r\psi(r)\to\infty$ as $r\to\infty$, which we do not need in this paper.

 We study
the existence and uniqueness of \eqref{e1.1} under two different sets of conditions requiring a dif\-fe\-rent functional approach. The first one, which will be presented in Section 3, assumes that $\psi$ is monotonically non\-decreasing and Lipschitz. The state space for \eqref{e1.1} is, in this case, $H^{-1}(\R^d)$, that is, the dual of the classical Sobolev space $H^1(\R^d)$.
 In spite of the apparent lack of generality ($\psi$ Lipschitz), it should be mentioned that there are physical models described by such an equation as, for instance, the two phase Stefan transition problem perturbed by a stochastic Gaussian noise \cite{1az}; moreover, in this latter case there is no restriction on the dimension $d$.

 The second case, which will be studied in Section 4, is that where $\psi$ is a maximal monotone multivalued function with at most polynomial growth. An important physical problem covered by this case is the self-organized cri\-ti\-ca\-lity model
 \begin{equation}
 \label{1.3}
 dX-\Delta H(X-X_c)dt=(X-X_c)dW,\end{equation}  where $H$ is the Heaviside function and $X_c$ is the critical state (see \cite{BDPR09}, \cite{5}, \cite{6}). More generally, this equation with discontinuous $\psi$ covers the stochastic nonlinear diffusion equation with singular diffusivity $D(u)=\psi'(u).$

 It should be mentioned that, in this second case, the solution $X(t)$ to \eqref{e1.1} is defined in a certain distribution space $\calh^{-1}$ (see \eqref{e2.2} below) on $\R^d$ and the existence is obtained for $d\ge3$ only, as in the case of continuous $\psi$ in \cite{Ren}. The case $1\le d\le2$ remains open due to the absence of a multiplier rule in the norm $\|\cdot\|_{\calh^{-1}}$ (see Lemma \ref{l4.3a} below).

 In Section 5, we prove  the finite time extinction of the solution $X$ to \eqref{e1.1} with strictly positive probability under the assumption that $\psi(r)r\ge\rho|r|^{m+1}$ and $m=\frac{d-2}{d+2}\,\cdot$

Finally, we would like to comment on one type of noise. Existence and uniqueness can be proved with $g(t,X(t))$ by replacing $X(t)$ under (more or less the usual) abstract conditions on $\sigma$ (see, e.g., \cite{Ren}, \cite{15prim}). The main reason why in this paper we restrict ourselves to linear multiplicative noise  is that first we want to be concrete, second the latter case is somehow generic (just think of taking the Taylor expansion of $\sigma(t,\cdot)$ up to first order), and third for this type of noise we prove finite time extinction in Section 5.

\section{Preliminaries}
\setcounter{equation}{0}

To begin with, let us briefly recall a few definitions pertaining distribution spaces on $\R^d$,
  whose  classical Euclidean norm will be denoted by $\vert \cdot \vert$.

Denote by $\cals'(\R^d)$ the space of all temperate distributions on $\R^d$ (see, e.g., \cite{8a}) and by $\calh$ the space
\begin{equation}\label{e2.1}
\calh=\{\vf\in\cals'(\R^d);\ \xi \mapsto  |\xi|\calf(\vf)(\xi)\in L^2(\R^d)\},\end{equation}
where $\calf(\vf)$ is the Fourier transform of $\vf$. We denote by
$L^2(\R^d)$ the space of square integrable functions on $\R^d$ with  norm   $|\cdot|_2$ and scalar pro\-duct $\<\cdot,\cdot\>_2$.
In general $\vert \cdot \vert _p$ will denote the norm of $L^p(\R^d)$ or $L^p(\R^d;\R^d), \ 1 \le p \le \infty$.
 The dual space $\calh\1$ of $\calh$ is given by
\begin{equation}\label{e2.2}
\calh\1=\{\eta\in\cals'(\R^d);\ \xi \mapsto
\calf(\eta)(\xi)|\xi|\1\in L^2(\R^d)\}.\end{equation}
The duality between $\calh$ and $\calh\1$ is denoted by $\<\cdot,\cdot\>$ and is given by
\begin{equation}\label{e2.3}
\<\vf,\eta\>=\int_{\R^d}\calf(\vf)(\xi)\ov{\calf(\eta)}(\xi)d\xi\end{equation}and the norm of $\calh$ denoted by $\|\cdot\|_1$ is given by
\begin{equation}\label{e2.4}
\|\vf\|_1=\(\int_{\R^d}|\calf(\vf)(\xi)|^2
|\xi|^2d\xi\)^{\frac12}=
\(\int_{\R^d}|\na\vf|^2d\xi\)^{\frac12}.\end{equation}
The norm of $\calh\1$, denoted by $\|\cdot\|_{-1}$ is given by
\begin{equation}\label{e2.5}
\|\eta\|_{-1}=\(\int_{\R^d}|\xi|^{-2}
|\calf(\eta)(\xi)|^2d\xi\)^{\frac12}=
\(\<(-\Delta)^{-1}\eta,\eta\>\)^{\frac12}.\end{equation}
(We note that the operator $-\Delta$ is an isomorphism from $\calh$ onto $\calh\1$.) The scalar product of $\calh\1$ is given by
\begin{equation}\label{e2.6}
\<\eta_1,\eta_2\>_{-1}=
\<(-\Delta)^{-1}\eta_1,\eta_2\>.\end{equation}As regards the relationship of $\calh$ with the space $L^p(\R^d)$ of $p$-summable functions on $\R^d$, we have the following.

\begin{lemma}\label{l1.1} Let $d\ge3$. Then we have
\begin{equation}\label{e2.7}
\calh\subset L^{\frac{2d}{d-2}}(\R^d)\end{equation}algebraically and topologically.\end{lemma}

Indeed, by the
Sobolev embedding theorem (see, e.g., \cite{7}, p. 278), we~have
$$|\vf|_{\frac{2d}{d-2}}\le C|\na \vf|_2,\ \ff\vf\in C^\9_0(\R^d),$$and, by density,
 this implies \eqref{e2.7}, as~claimed.

It should be mentioned that \eqref{e2.7} is no longer true for $1\le d\le2$. \mbox{However,} by duality, we have
\begin{equation}\label{e2.8}
L^{\frac{2d}{d+2}}(\R^d)\subset\calh\1,\ \ff d\ge3.\end{equation}

Denote by $H^1(\R^d)$ the Sobolev space
$$\barr{lcl}
H^1(\R^d)&=&\{u\in L^2(\R^d);\ \na u\in L^2(\R^d)\}\vsp
&=&\{u\in L^2(\R^d);\ \xi \mapsto \calf(u)(\xi)
 (1+ |\xi|^2)^{\frac12} \in L^2(\R^d)\}\earr$$with norm
 $$|u|_{H^1(\R^d)}=\(\dd\int_{\R^d}(u^2+|\na u|^2)d\xi\)^{\frac12}
 =\(\dd\int_{\R^d}|\calf u(\xi)|^2(1+|\xi|^2)d\xi\)^{\frac12}$$
and by $H\1(\R^d)$ its dual, that is,
$$H\1(\R^d)=\{u\in\cals'(\R^d);\ \calf(u)(\xi)(1+|\xi|^2)^{-\frac12}\in L^2(\R^d)\}.$$
The norm of $H\1(\R^d)$ is denoted by $|\cdot|_{-1}$ and its scalar product by $\left<\cdot,\cdot\right>_{-1}$. We have the continuous and dense embeddings
$$H^1(\R^d)\subset\calh,\ \calh\1\subset H\1(\R^d).$$ It should be emphasized, however, that $\calh$ is not a subspace of $L^2(\R^d)$ and so $L^2(\R^d)$ is not the pivot space in the duality $\<\cdot,\cdot\>$ given by \eqref{e2.3}.

Given a Banach space $Y$, we denote by $L^p(0,T;Y)$ the space of all $Y$-valued $p$-integrable functions on $(0,T)$ and by $C([0,T];Y)$ the space of continuous $Y$-valued functions on $[0,T]$. For two Hilbert spaces $H_1,H_2$ let $L(H_1,H_2)$ and $L_2(H_1,H_2)$ denote the set of all bounded linear and Hilbert-Schmidt operators, respectively. We refer to \cite{dpz}, \cite{rockpre} for definitions and basic results pertaining infinite dimensional stochastic processes.

\section{Equation \eqref{e1.1} with the Lipschitzian $\psi$}
\setcounter{equation}{0}

Consider here equation \eqref{e1.1} under the \fwg\ conditions.
\bit\item[(i)] $\psi:\R\to\R$ is monotonically nondecreasing,  Lipschitz
 such that \mbox{$\psi(0)=0$.}
\item[(ii)] $W$ is a Wiener  process as in \eqref{e1.2},
where $e_k\in H^1(\R^d)$, such that
\begin{equation}\label{e3.1}
C^2_\infty:=36\sum^\9_{k=1}\mu^2_k(|\na e_k|^2_{\9}+|e_k|^2_\9+1)<\9,\end{equation}and $\{e_k\}$ is an orthonormal basis in $H^{-1}(\R^d)$.
\eit
We insert the factor $36$ for convenience here to avoid additional large numerical constants in subsequent estimates.

\begin{remark} \label{r3.0} {\rm By Lemma \ref{l4.3a} below, $|\nabla e_k|_\infty$ in \eqref{e3.1} can be replaced by $|\nabla e_k|_d$, and all the results in this section remain true.}\end{remark}

\begin{definition}\label{d3.1} {\rm Let $x\in H\1(\R^d)$. A continuous, $(\calf_t)_{t\ge0}$-adapted process $X:[0,T] \to H\1(\R^d)$  is called  strong solution  to \eqref{e1.1} if the \fwg\ conditions hold:
\begin{eqnarray}
&&X\in L^2(\ooo;C([0,T];H\1(\R^d)))\cap L^2([0,T]\times\Omega;L^2(\R^d))\label{e3.2}\\
&&\int^\bullet_0\psi(X(s))ds\in C([0,T]; H^1(\R^d)),\ \pas\label{e3.3}\\
&&X(t)-\Delta\int^t_0\psi(X(s))ds=x+\int^t_0X(s)dW(s), \  \ff t\in[0,T], \ \pas \label{e3.4} 
% &&\hspace*{55mm}  \ff t\in[0,T],\ \pas\nonumber
\end{eqnarray}}\end{definition}

\begin{remark}\label{r3.1prim} {\rm The stochastic (It\^o-) integral in \eqref{e3.4} is the standard one from \cite{dpz} or \cite{rockpre}. In fact, in the terminology of these references, $W$ is a $Q$-Wiener process $W^Q$ on $H^{-1}$, where $Q:H^{-1}\to H^{-1}$ is the symmetric trace class operator defined by
$$Qh:=\sum^\9_{k=1}\mu^2_k\<e_k,h\>_{-1}e_k,\ h\in H\1.$$
For $x\in H\1$, define $\sigma(x):Q^{1/2}H\1\to H\1$ by
\begin{equation}\label{e3.4prim}
\sigma(x)(Q^{1/2}h)=\sum^\9_{k=1}(\mu_k\<e_k,h\>_{-1}e_k\cdot x),\ h\in H.\end{equation}
By \eqref{e3.1}, each $e_k$ is an $H\1$-multiplier such that
\begin{equation}\label{e3.4secund}
|e_k\cdot x|_{-1}\le2\(|e_k|_\9+|\nabla e_k|_\9\)|x|_{-1},\ x\in H\1. \end{equation}
Hence, for all $x\in H\1,\ h\in H\1,$
$$\barr{ll}
\dd\sum^\9_{k=1}\left|\mu_k\<e_k,h\>_{-1}e_k x\right|_{-1}
&\le\(\dd\sum^\9_{k=1}\mu^2_k|e_k x|^2_{-1}\)^{1/2}|h|_{-1}\vsp
&\le2C_\9|x|_{-1}|h|_{-1}\vsp
&=2C_\9|x|_{-1}|Q^{1/2}h|_{Q^{1/2}H\1},\earr$$
and thus $\sigma(x)$ is well-defined and an element in $L(Q^{1/2}H\1,H\1)$. Moreover, for $x\in H\1$, by \eqref{e3.4prim}, \eqref{e3.4secund},
\begin{equation}\label{e3.4tert}
\barr{ll}
\|\sigma(x)\|^2_{L_2(Q^{1/2}H\1,H\1)}
&=\dd\sum^\9_{k=1}|\sigma(x)(Q^{1/2}e_k)|^2_{-1}
=\dd\sum^\9_{k=1}|\mu_ke_kx|^2_{-1}\vsp
&=\dd\sum^\9_{k=1}\mu^2_k|e_kx|^2_{-1}
\le C^2_\9|x|^2_{-1}.\earr\end{equation}
Since $\{Q^{1/2}e_k\mid k\in\mathbb{N}\}$ is an orthonormal basis of $Q^{1/2}H\1$, it follows that $\sigma(x)\in L_2(Q^{1/2}H\1,H\1)$ and the map $x\mapsto\sigma(x)$ is linear and continuous (hence Lipschitz) from $H\1$ to $L_2(Q^{1/2}H\1,H\1)$. Hence (e.g., according to \cite[Section~2.3]{rockpre})
$$\int^t_0X(s)dW(s):=
\int^t_0\sigma(X(s))dW^Q(s),\ t\in[0,T],$$is well-defined as a continuous $H\1$-valued martingale and by It\^o's isometry and \eqref{e3.4tert}
\begin{equation}\label{e3.7prim}
 \barr{lcl}
 \E\left|\dd\int^t_0 X(s)dW(s)\right|^2_{-1}&
 =&\dd\sum^\9_{k=1}\mu^2_k\E
 \int^t_0|X(s)e_k|^2_{-1}ds\\
 &\le & C^2_\9\E\dd\int^t_0\!|X(s)|^2_{-1}ds,\  t\in[0,T].\earr\end{equation}
 Furthermore,  it follows that
\begin{equation}\label{e3.7secund}
\barr{lcl}
\dd\int^t_0X(s)dW(s)&=&
\dd\sum^\9_{k=1}\int^t_0\sigma(X(s))(Q^{1/2}e_k)
d\beta_k(s)\vsp
&=&\dd\sum^\9_{k=1}\int^t_0\mu_ke_kX(s)d\beta_k(s),\ t\in[0,T],\earr\end{equation}
where the series converges in $L^2(\ooo;C([0,T];H\1)).$

In fact, $\int^\bullet_0X(s)dW(s)$ is a continuous $L^2$-valued martingale, because  $X\in L^2([0,T]\times\Omega;L^2(\R^d))$ and, analogously to \eqref{e3.4tert}, we get
$$\|\sigma(x)\|^2_{L_2(Q^{1/2}H\1,L^2)}\le C^2_\9|x|^2_2,\ \ x\in L^2(\R^d).$$
In particular, by It\^o's isometry,
$$\E\left|\int^t_0 X(s)dW(s)\right|^2_2\le C^2_\9\E\dd\int^t_0|X(s)|^2_2ds,\ t\in[0,T].$$
Furthermore, the series in \eqref{e3.7prim} even converges in $L^2(\Omega;C([0,T];L^2(\R^d))).$

We shall use the facts presented in this remark throughout this paper without further notice.}\end{remark}

\begin{theorem}\label{t3.1} Let $d\ge1$ and $x\in L^2(\R^d)$. Then, under assumptions {\rm(i), (ii)}, there is a unique strong solution  to equation \eqref{e1.1}. This solution satisfies
$$\E\left[\sup_{t\in[0,T]}|X(t)|^2_2\right]
\le 2|x|^2_2e^{3C^2_\infty t}.$$In particular, $X\in L^2(\Omega;L^\infty([0,T];L^2(\R^d)))$. Assume further that
\begin{equation}\label{e3.5}
\psi(r)r\ge\a r^2,\ \ff r\in\R,\end{equation}where
$\a>0$. Then, there is a unique strong solution $X$ to \eqref{e1.1} for all $x\in H\1(\R^d).$\end{theorem}

\n{\bf Proof of Theorem \ref{t3.1}.} We approximate \eqref{e1.1} by
\begin{equation}\label{e3.6}
\barr{l}
dX+(\nu-\Delta)\psi(X)dt=XdW(t),\ t\in(0,T),\\
X(0)=x\ \onr\ \R^d,\earr\end{equation}where $\nu\in(0,1)$. We have
the following.

\begin{lemma}\label{l3.2} Assume that $\psi$ is as in assumption {\rm(i)}.   Let $x\in L^2(\R^d)$. Then, there is a unique $(\calf_t)_{t\ge0}$-adapted solution $X=X^\nu$ to \eqref{e3.6} in the \fwg\ strong sense:
\begin{equation}
X^\nu\in L^2(\ooo,C([0,T];H\1(\R^d)))\cap L^2([0,T]\times\Omega;L^2(\R^d)), \label{e3.7}\end{equation}
and $\pas$
\begin{equation}
X^\nu(t)=x+(\Delta-\nu)\int^t_0\psi(X^\nu(s))ds
 +\int^t_0 X^\nu(s) dW(s),\label{e3.10}
 \ t\in[0,T].
\end{equation}
In addition, for all $\nu\in(0,1)$,
 \begin{equation}\label{e3.12prim}
 \E\left[\sup_{t\in[0,T]}|X^\nu(t)|^2_2\right]\le2|x|^2_2 e^{3C^2_\9T}.\end{equation}
 If, moreover, $\psi$ satisfies \eqref{e3.5}, then for each $x\in H\1(\R^d)$ there is a unique solution $X^\nu$ satisfying \eqref{e3.7}, \eqref{e3.10}. \end{lemma}

%%%%%%%%%

\n{\bf Proof of Lemma \ref{l3.2}.} Let us start with the second part of the assertion, i.e., we assume that $\psi$ satisfies \eqref{e3.5} and that $x\in H\1(\R^d)$. Then the standard theory  (see, e.g., \cite[Sections 4.1 and 4.2]{rockpre}) applies to ensure that there exists a unique solution $X^\nu$ taking value in $H\1(\R^d)$ satisfying \eqref{e3.7}, \eqref{e3.10} above. Indeed, it is easy to check that (H1)--(H4) from \cite[Section 4.1]{rockpre} are satisfied with $V:=L^2(\R^d)$, $H:=H\1(\R^d)$, $Au:=(\Delta-\nu)(\psi(u)),$ $u\in V$, and $H\1(\R^d)$ is equipped with the equivalent norm
$$|\eta|_{-1,\nu}:=\<\eta,(\nu-\Delta)\1\eta\>^{1/2},\ \eta\in H\1(\R^d),$$
(in which case, we also write $H\1_\nu)$. Here, as before, we use $\left<\cdot,\cdot\right>$ also to denote the dualization between $H^1(\R^d)$ and $H\1(\R^d)$. For details, we refer to the calculations in \cite[Example 4.1.11]{rockpre}, which  because $p=2$ go through when the bounded domain $\Omega$ there is replaced by $\R^d$. Hence \cite[Theorem 4.2.4]{rockpre} applies to give the above solution $X^\nu$.

In the case when $\psi$ does not satisfy \eqref{e3.5}, the above conditions (H1), (H2), (H4) from \cite{rockpre} still hold, but (H3) not in general.  Therefore, we replace $\psi$ by $\psi+\lambda I$, $\lambda\in(0,1)$, and thus consider $A_\lbb(u):=(\Delta-\nu)(\psi(u)+\lbb u)$, $u\in V:=L^2(\R^d)$ and, as above, by \cite[Theorem 4.2.4]{rockpre}, obtain a solution $X^\nu_\lbb$, satisfying \eqref{e3.7}, \eqref{e3.10}, to
\begin{equation}\label{e3.13z}
\!\barr{l}
dX^\nu_\lbb(t)+(\nu-\Delta)
(\psi(X^\nu_\lbb(t))+\lbb X^\nu_\lbb(t))dt
=X^\nu_\lbb(t)dW(t),\ t\in[0,T],\vsp
X^\nu_\lbb(0)=x\in H\1(\R^d).
\earr\!
\end{equation}In particular, by \eqref{e3.7},
\begin{equation}\label{e3.14z}
\E\left[\sup_{t\in[0,T]}
|X^\nu_\lbb(t)|^2_{-1}\right]<\9.\end{equation}
We want to let $\lbb\to0$ to obtain a solution to \eqref{e3.6}.  To  this end, in this case (i.e., without assuming \eqref{e3.5}), we     assume from now on that   $x\in L^2(\R^d)$.   The reason is that we need the following.

\bk\n{\bf Claim 1.} {\it We have $X^\nu_\lbb\in L^2([0,T]\times\Omega;H^1(\R^d))$ and $$\barr{r}
\E\left[\dd\sup_{t\in[0,T]}|X^\nu_\lbb(t)|^2_2\right]+4\lbb\E\dd\int^T_0|\na X^\nu_\lbb(s)|^2_2ds\le2|x|^2_2
e^{3C^2_\9T},\vsp\mbox{ for all $\nu,\lbb\in(0,1).$}\earr$$}

\n Furthermore, $X^{\nu}_\lbb$ has continuous sample paths in $L^2(\R^d)$, $\pas$

\bk\n{\bf Proof of Claim 1.} We know that
\begin{equation}\label{e3.15z}
% \barr{r}
X^\nu_\lbb(t)=x+(\Delta-\nu)\dd\int^t_0(\psi(X^\nu_\lbb(s))
+\lbb X^\nu_\lbb(s))ds 
%\vsp
+
\dd\int^t_0X^\nu_\lbb(s)dW(s),\ t\in[0,T].
%\earr
\end{equation}
Let $\a\in(\nu,\9)$. Recalling that $(\a-\Delta)^{-\frac12}:H\1(\R^d)\to L^2(\R^d)$ and applying this operator to the above equation, we find
\begin{equation}\label{e3.16z}
\barr{lcl}
(\a-\Delta)^{-\frac12}X^\nu_\lbb(t)\vsp
=(\a-\Delta)^{-\frac12}x+\dd\int^t_0(\Delta-\nu)(\a-\Delta)^{-\frac12}
(\psi(X^\nu_\lbb(s))+\lbb X^\nu_\lbb(s))ds\vsp
+\dd\int^t_0(\a-\Delta)^{-\frac12}\sigma(X^\nu_\lbb(s))
Q^{1/2}dW(s),\ t\in[0,T].\earr
\end{equation}
Applying It\^o's formula (see, e.g., \cite[Theorem 4.2.5]{rockpre} with $H=L^2(\R^d)$) to $|(\a-\Delta)^{-\frac12}X^\nu_\lbb(t)|^2_2$, we obtain, for $t\in[0,T]$,

\begin{equation}\label{e3.17z}
\barr{ll}
|(\a-\Delta)^{-\frac12} X^\nu_\lbb(t)|^2_2
=|(\a-\Delta)^{-\frac12}x|^2_2\vsp
+2\dd\int^t_0\<(\Delta-\nu)
(\alpha-\Delta)^{-\frac12}\psi(X^\nu_\lbb(s)),
(\a-\Delta)^{-\frac12}X^\nu_\lbb(s)\>ds\vsp
-2\lbb\dd\int^t_0(|\na(
(\a-\Delta)^{-\frac12}X^\nu_\lbb(s))|^2_2+
\nu|(\a-\Delta)^{-\frac12}X^\nu_\lbb(s)|^2_2)ds\vsp
+\dd\int^t_0\|(\a-\Delta)^{-\frac12}\sigma(X^\nu_\lbb(s))
Q^{1/2}\|^2_{L_2(H\1,L^2)}ds\vsp
+2\dd\int^t_0\<(\a-\Delta)^{-\frac12}X^\nu_\lbb(s),
(\a-\Delta)^{-\frac12}
\sigma(X^\nu_\lbb(s))Q^{1/2}dW(s)\>_2.
\earr\end{equation}But, for $f\in L^2(\R^d)$, we have $$(\a-\Delta)^{-\frac12}(\Delta-\nu)(\a-\Delta)^{-\frac12}f=(P-I)f,$$where$$P:=(\a-\nu)(\a-\Delta)\1.$$
For the Green function $g_\a$ of $(\a-\Delta)$, we then have, for $f\in L^2(\R^d)$,
$$Pf=(\a-\nu)\int_{\R^d}f(\xi)g_\a(\cdot,\xi)d\xi.$$
Hence, by  \cite[Lemma 5.1]{15prim},  the integrand of the second term on the right-hand side of \eqref{e3.17z} with $f:=X^\nu_\lbb(s)$ $(\in L^2(\R^d)$ for $ds$-a.e. $s\in[0,T])$ can be rewritten as
$$\barr{lcl}
\<\psi(f),(P-I)f\>_2&\!\!\!=\!\!\!\!&
-\dd\frac12\int_{\R^d}\int_{\R^d}
[\psi(f(\wt\xi)){-}\psi(f(\xi))]
[f(\wt\xi){-}
f(\xi)]g_\a(\xi,\wt\xi)d\wt\xi\,d\xi\vsp
&&-\dd\int_{\R^d}(1-P1(\xi))\cdot\psi(f(\xi))f(\xi)d\xi.\earr$$
Since $\psi$ is monotone, $\psi(0)=0$ and $P1\le1$, we deduce that
$$\<\psi(f),(P-I)f\>\le 0.$$
Hence, after a multiplication by $\a$, \eqref{e3.17z} implies that, for all $t\in[0,T]$ (see~Remark \ref{r3.1prim}),

$$\barr{l}
\a|(\a-\Delta)^{-\frac12} X^\nu_\lbb(t)|^2_2 +
2\lbb\dd\int^t_0|\na(\sqrt\a
(\a-\Delta)^{-\frac12} X^\nu_\lbb(s))|^2_2ds\\
\le\a|(\a-\Delta)^{-\frac12}x|^2_2+\dd\int^t_0
\dd\sum^\9_{k=1}\mu^2_k
\<\a(\a-\Delta)^{-1}(e_kX^\nu_\lbb(s)),
e_kX^\nu_\lbb(s)\>_2ds\\
+2\dd\int^t_0\<\a(\a-\Delta)\1X^\nu_\lbb(s),\sigma(X^\nu_\lbb(s))
Q^{1/2}dW(s)\>_2.\earr$$
Hence, by the Burkholder--Davis--Gundy (BDG) inequality (with $p=1$) and since $\a(\a-\Delta)\1$ is a contraction on $L^2(\R^d)$,
\begin{equation}\label{e3.18z}
\barr{l}
\E\left[\dd\sup_{s\in[0,t]}|\sqrt\a(\a-\Delta)^{-\frac12}X^\nu_\lbb(s)|^2_2\right]
\\
\qquad\quad+2\lbb\E\dd\int^t_0
|\na(\sqrt\a(\a-\Delta)^{-\frac12}
X^\nu_\lbb(s))|^2_2ds\\
\qquad\quad\le|\sqrt\a (\a-\Delta)^{-\frac12}x|^2_2+C^2_\9
\E\dd\int^t_0|X^\nu_\lbb(s)|^2_2ds\\
\qquad\quad+6\E\(\dd\int^t_0\sum^\9_{k=1}\mu^2_k\<\a(\a-\Delta)\1X^\nu_\lbb(s),
e_kX^\nu_\lbb(s)\>^2_2ds\)^{1/2}.\earr\end{equation}
The latter term can be estimated by
\begin{equation}\label{e3.19z}
\barr{l}
C_\9\E\left[\dd\sup_{s\in[0,t]}|\a(\a-\Delta)\1X^\nu_\lbb(s)|_2
\(\dd\int^t_0|X^\nu_\lbb(s)|^2_2ds\)^{1/2}\right]\vsp
\le\dd\frac12\,\E\left\{\dd\sup_{s\in[0,t]}
|\sqrt\a(\a-\Delta)^{-\frac12}X^\nu_\lbb(s)|^2_2\right]
+\dd\frac12\,C^2_\9\E\dd\int^t_0|X^\nu_\lbb(s)|^2_2ds,\earr\end{equation}
where we used that $\sqrt\a(\a-\Delta)^{-\frac12}$
is a contraction on $L^2(\R^d)$.
Note that the first summand on the right-hand side is finite by \eqref{e3.14z}, since the norm $|\sqrt{\a}(\a-\Delta)^{-\frac12}\cdot|_2$ is equivalent to $|\cdot|_{-1}$. Hence, we can subtract this term after substituting \eqref{e3.19z} into \eqref{e3.18z} to obtain
\begin{equation}\label{e3.20z}
\barr{l}
\E\left[\dd\sup_{s\in[0,t]}|\sqrt\a
(\a-\Delta)^{-\frac12}X^\nu_\lbb(s)|^2_2\right]
\\\qquad\qquad
+4\lbb\E\dd\int^t_0
|\na(\sqrt\a(\a-\Delta)^{-\frac12}X^\nu_\lbb(s)
|^2_2ds\\
\qquad\qquad \le2|\sqrt\a(\a-\Delta)^{-\frac12}x|^2_2+3C^2_\9\E\dd\int^t_0|X^\nu_\lbb(s)|^2_2ds,\ t\in[0,T].\earr\end{equation}
Obviously, the quantity under the $\dd\sup_{s\in[0,t]}$ on the left-hand side of \eqref{e3.20z} is increasing in $\a$. So, by the monotone convergence theorem, we may let $\a\to\9$ in \eqref{e3.20z} and then, except for its last part, Claim 1 immediately follows by Gronwall's lemma, since $\sqrt{\a}(\a-\Delta)^{-\frac12}$ is a contraction in $L^2(\R^d)$ and  $x\in L^2(\R^d)$. The last part of Claim 1 then immediately follows from \mbox{\cite[Theorem 2.1]{13a}. \rule{1,5mm}{1,5mm}}\bk

Applying It\^o's formula to $|X^\nu_\lbb(t)-X^\nu_{\lbb'}(t)|^2_{-1,\nu}$ (see \cite[Theorem 4.2.5]{rockpre}), it follows from \eqref{e3.15z} that, for $\lbb,\lbb'\in(0,1)$ and $t\in[0,T]$,
\begin{equation}\label{e3.21z}
\barr{l}
|X^\nu_\lbb(t)-X^\nu_{\lbb'}(t)|^2_{-1,\nu}\vsp
\quad+2\dd\int^t_0\<\psi(X^\nu_\lbb)-\psi(X^\nu_{\lbb'})
+(\lbb X^\nu_\lbb-\lbb' X^\nu_{\lbb'}),
X^\nu_\lbb-X^\nu_{\lbb'}\>_2ds\vsp
\quad=\dd\int^t_0\|\sigma(X^\nu_\lbb(s)-X^\nu_{\lbb'}(s)\|^2_{L_2(Q^{1/2}H\1,H^{-1}_\nu)}ds\vsp
\quad+2\dd\int^t_0\<X^\nu_{\lbb}(s)-X^\nu_{\lbb'}(s),
\sigma(X^\nu_\lbb(s)-X^\nu_{\lbb'}(s))dW^Q(s)\>_{-1,\nu}.\earr\end{equation}
Our assumption (i) on $\psi$ implies that
$$(\psi(r)-\psi(r'))(r-r')\ge({\mathrm Lip}\,\psi+1)\1|\psi(r)-\psi(r')|^2,\mbox{ for $r,r'\in\R,$}$$
where ${\mathrm Lip}\,\psi$ is  the Lipschitz constant of $\psi.$ Hence \eqref{e3.21z}, \eqref{e3.4tert} and the BDG inequality (for $p=1$ imply that, for all $t\in[0,T]$)
$$\barr{l}
\E\left[\dd\sup_{s\in[0,t]}
|X^\nu_\lbb(s)-
X^\nu_{\lbb'}(s)|^2_{-1,\nu}\right]\vsp
 +2({\mathrm Lip}\,\psi+1)\1
\E\dd\int^t_0|\psi(X^\nu_\lbb(s))-
\psi(X^\nu_{\lbb'}(s))|^2_2ds\vsp
\le2(\lbb+\lbb')\E\dd\int^t_0
(|X^\nu_\lbb(s)|^2_2+|X^\nu_{\lbb'}(s)|^2_2)ds
+C^2_\9\dd\int^t_0
|X^\nu_\lbb(s)-X^\nu_{\lbb'}(s)|^2_{-1,\nu}ds\vsp
+2\E\(\dd\int^t_0\sum^\9_{k=1}\mu^2_k
\<X^\nu_\lbb(s)-X^\nu_{\lbb'}(s),
(X^\nu_\lbb(s)-X^\nu_{\lbb'}(s))
e_k\>^2_{-1,\nu}ds\)^{1/2}.\earr$$
By \eqref{e3.4tert} and Young's inequality, the latter term is  dominated by
$$\frac12\,
\E\left[\dd\sup_{s\in[0,t]}|X^\nu_\lbb(s)
-X^\nu_{\lbb'}(s)|^2_{-1,\nu}\right]
+
\dd\frac12\,C^2_\9\E\dd\int^t_0|X^\nu_
\lbb(s)-X^\nu_{\lbb'}(s)|^2_{-1,\nu}ds.$$
Hence, because of $x\in L^2(\R^d)$ and Claim 1, we may now apply Gronwall's lemma to obtain that, for some constant $C$ independent of $\lbb',\lbb$ (and $\nu$),
\begin{equation}\label{e3.22z}
%\barr{l}
\E\left[\dd\sup_{t\in[0,T]}|X^\nu_\lbb(t)-X^\nu_{\lbb'}(t)|^2_{-1,\nu}\right]
+
\E\dd\int^T_0|\psi(X^\nu_\lbb(s))-\psi(X^\nu_{\lbb'}(s))|^2_2ds\le C(\lbb+\lbb').
%\earr
  \end{equation}
Hence there exists an $(\calf_t)$-adapted continuous $H\1$-valued process $X^\nu=(X^\nu(t))_{t\in[0,T]}$ such that $X^\nu\in L^2(\Omega;C([0,T];H\1))$. Now, by Claim 1, it follows  that
$$X^\nu\in L^2([0,T]\times\ooo;L^2(\R^d)).$$

\n{\bf Claim 2.} {\it $X^\nu$ satisfies equation \eqref{e3.10} $($i.e.,   we can pass to the limit in \eqref{e3.15z} as $\lbb\to0)$.}

\bk\n{\bf Proof of Claim 2.} We already know that
$$X^\nu_\lbb\longrightarrow X^\nu\mbox{\ \ and\ \ }
\int_0^\bullet X^\nu_\lbb(s)dW(s)\longrightarrow\int^\bullet_0X^\nu(s)dW(s)$$
in $L^2(\ooo;C([0,T];H\1))$ as $\lbb\to0$ (for the second convergence see the above argument using \eqref{e3.4tert} and the BDG inequality). So, by \eqref{e3.15z} it follows that
$$\int^\bullet_0(\psi(X^\nu_\lbb(s))+\lbb X^\nu_\lbb(s)))ds,\ \lbb>0,$$
converges as $\lbb\to0$ to an element in $L^2(\ooo;C([0,T];H^1).$ But, by \eqref{e3.22z} and Claim 1, it follows that
\begin{equation}\label{e3.23z}
\int^\bullet_0(\psi(X^\nu_\lbb(s))+\lbb X^\nu_\lbb(s))ds\longrightarrow\int^\bullet_0\psi(X^\nu(s))ds\end{equation}
as $\lbb\to0$ in $L^2(\ooo;L^2([0,T];L^2(\R^d)))$. Hence Claim 2 is proved. $\Box$\bk

Now, \eqref{e3.12prim}  follows from Claim 1 by lower semicontinuity. This completes the proof of Lemma \ref{l3.2}. $\Box$

%%%%%%%%%%%%%%%

\bk\n{\bf Proof of Theorem \ref{t3.1} (continued).} We are going to use Lemma \ref{l3.2} and let $\nu\to0$. The arguments are similar to those in the proof of Lemma \ref{l3.2}. So, we shall not repeat all the details.

Now, we rewrite \eqref{e3.6} as
\begin{equation} \label{e319prime}
dX^\nu+(I-\Delta)\psi(X^\nu)dt=(1-\nu)\psi(X^\nu)dt+X^\nu dW(t)
\end{equation}
and apply It\^o's formula to $\vf(x)=\frac12\,|x|^2_{-1}$ (see, e.g., \cite[Theorem 4.2.5]{rockpre}). We~get, for $x\in H^{-1}$, by \eqref{e3.7prim} and after taking expectation,

$$\barr{l}
\dd\frac12\,\E|X^\nu(t)|^2_{-1}+
\E\int^t_0\int_{\R^d}
\psi(X^\nu(s))X^\nu(s)d\xi\,ds\vsp
\qquad=\dd\frac12\,|x|^2_{-1}+(1-\nu)
\E\dd\int^t_0
 \<\psi(X^\nu(s)),X^\nu(s)\>_{-1}ds
 \vsp
\qquad+\dd\frac12\,\E\int^t_0 \sum^\9_{k=1}\mu^2_k|X^\nu e_k|^2_{-1}ds\vsp
\qquad\le\dd\frac12\,|x|^2_{-1}+\E \int^t_0 |\psi(X^\nu)|_{-1}|X^\nu|_{-1}ds
\vsp
\qquad+\dd\frac12C^2_\9\E \dd\int^t_0 |X^\nu(s)|^2_{-1}ds,\ \ff t\in[0,T].
\earr$$

Recalling that $|\cdot|_{-1}\le|\cdot|_2$,
  we get, via Young's and Gronwall's inequalities, for some $C\in(0,\9)$ that
\begin{equation}\label{e3.25z}
\E|X^\nu(t)|^2_{-1}+\frac\alpha2\,\E\int^T_0|X^\nu(s)|^2_2ds
\le C|x|^2_{-1},\ t\in[0,T],\ \nu\in(0,1),\end{equation}because, by assumption (i),  $\psi(r)r\ge\wt\a|\psi(r)|^2$, $\ff r\in\R$, with $\wt \a:=\mbox{$({\mathrm Lip}\,\psi+1)\1$}$. Here we set $\a=0$ if \eqref{e3.5} does not hold.

Now, by a similar calculus, for $X^\nu-X^{\nu'}$ we get
$$\barr{l}
|X^\nu(t)-X^{\nu'}(t)|^2_{-1}
+2\dd\int^t_0\int_{\R^d}
(\psi(X^\nu)-\psi(X^{\nu'}))(X^\nu-X^{\nu'})
d\xi\,ds\vsp
\quad\le C\dd\int^t_0\<\psi(X^\nu)-\psi(X^{\nu'}),
X^\nu-X^{\nu'}\>_{-1} ds \vsp
\quad+C\dd\int^t_0(\nu|\psi(X^\nu)|_2+\nu'|\psi(X^{\nu'})|_2)
|X^\nu-X^{\nu'}|_{-1}ds\vsp
\quad+C\dd\int^t_0|X^\nu-X^{\nu'}|^2_{-1}ds
+\dd \sum^\9_{k=1} \int_0^t \mu_k\<(X^\nu-X^{\nu'}),e_k(X^\nu-X^{\nu'})
\>_{-1}d\b_k,\vsp
\hfill t\in[0,T].\earr$$
Taking into account that, by assumption (i),  $$(\psi(x)-\psi(y))(x-y)\ge\wt\a|\psi(x)-\psi(y)|^2,\ \ff x,y \in \R^d,$$ we get, for all $\nu,\nu'>0,$
$$\barr{l}
\dd\vert X^\nu(t) - X^{\nu'}(t)\vert^2_{-1} + {\tilde \alpha}
 \int_0^t \vert \psi(X^\nu(s)) - \psi(X^{\nu'}(s))\vert^2_2 ds\vsp
 \dd\le  C_1  \int_0^t \vert X^\nu(s) - X^{\nu'}(s)\vert^2_{-1} ds +
\frac{\tilde \alpha}{2} \int_0^t \vert \psi(X^\nu(s)) -
 \psi(X^{\nu'}(s))\vert^2_2 ds \vsp
\dd+  C_2 (\nu +\nu')
 \dd\int^t_0(|\psi(X^\nu(s))|^2_2+
|\psi(X^{\nu'}(s))|^2_2)ds\vsp
\dd+  \sum_{k=1}^\infty  \int_0^t \mu_k
\langle (X^\nu(s) - X^{\nu'}(s)),e_k (X^\nu(s) - X^{\nu'}(s))\rangle_{-1} d \beta_k(s),\  t\in[0,T].
\earr$$
So, similarly to showing \eqref{e3.22z} in the proof of Lemma  \ref{l3.2}, by \eqref{e3.12prim}, if $x\in L^2(\R^d)$, and by \eqref{e3.25z}, if $x\in H\1(\R^d)$ and $\psi$ satisfies \eqref{e3.5}, by the Burkholder-Davis-Gundy  inequality, for $p=1$, we get, for all $\nu,\nu'\in(0,1)$,
$$
\E\sup_{t\in[0,T]}|X^\nu(t)-X^{\nu'}(t)|^2_{-1}
+\E\dd\int^T_0|\psi(X^\nu(s))-
\psi(X^{\nu'}(s))|^2_2ds
\le C(\nu+\nu').
$$
The remaining part of the proof is now exactly the same as the last part of the proof of Lemma \ref{l3.2}. $\Box$

\begin{remark}\label{r3.1} {\rm Theorem \ref{t3.1} is a basic tool for the probabilistic (double) re\-pre\-sentation of equation \eqref{e1.1}, which holds when $\psi$ is Lipschitz, as it is proved in \cite{9z}. If \eqref{e1.1} is not perturbed by noise, and $\psi$ is possibly discontinuous, 
its probabilistic representation was performed in  \cite{BRR1}, \cite{BRR2}, \cite{9z} with extensions and numerical simulations located in 
\cite{BR}, \cite{12z}.}\end{remark}

\section{Equation \eqref{e1.1} for maximal monotone  functions  $\psi$ with polynomial growth}

\label{s4}

\setcounter{equation}{0}

In this section, we assume $d\ge3$ and we shall study  the existence for equation \eqref{e1.1} under the \fwg\ assumptions:
\bit
\item[(j)] $\psi:\R\to2^\R$ is a maximal monotone graph such that $0\in\psi(0)$ and
\begin{equation}\label{e4.1}
\sup\{|\eta|;\ \eta\in\psi(r)\}\le C(1+|r|^m),\ \ff r\in\R,\end{equation}
where $1 \le m <\9$.
\item[(jj)] $W(t)=\dd\sum^\9_{k=1}\mu_k e_k\b_k(t),\ t\ge0$,
where $\{\b_k\}^\9_{k=1}$ are   independent Brownian motions on a stochastic basis
 $\{\ooo,\calf,\calf_t,\mathbb{P}\}$, $\mu_k\in\R$, and
$e_k\in C^1(\R^d)\cap \calh\1$ are such that $\{e_k\}$ is an  orthonormal basis in $\calh\1$ and
\begin{equation}\label{e4.2}
\sum^\9_{k=1}\mu^2_k(|e_k|^2_\9+|\na e_k|^2_d
 +1)<\9.\end{equation}
\eit
The existence of $\{e_k\}$ as in (jj) is ensured by the following lemma.

\begin{lemma}\label{l4.3a} Let $d\ge3$ and let $e\in L^\9(\R^d;\R^d)$ be such that $\na e\in L^{d}(\R^d;\R^d)$. Then
\begin{equation}\label{e4.12}
\|x e\|_{-1}\le\|x\|_{-1}(|e|_\9+C|\na e|_{d}),\ \ff x\in\calh\1,\end{equation}where $C$ is independent of $x$ and $e$.\end{lemma}

\begin{proof} We have
\begin{equation}\label{e4.13}
\|xe\|_{-1}=\sup\{\<x,e\vf\>;\|\vf\|_1\le1\}\le
\|x\|_{-1}\sup\{\|e\vf\|_1;\|\vf\|_1\le1\}.\end{equation}
On the other hand, by Lemma \ref{l1.1} we have,
for all $\vf\in C^\9_0(\R^d)$,
$$\barr{ll}
\|e\vf\|_{1}\!\!\!&\le |e\na\vf+\vf\na e|_2\le|e\na\vf|_2+|\vf\na e|_2\vsp
&\le |e|_\9|\na\vf|_2+|\vf|_{p}|\na e|_{d}\le|e|_\9\|\vf\|_1+C\|\vf\|_1|\na e|_{d},\earr$$where $p=\frac{2d}{d-2}\,\cdot$ Then, by \eqref{e4.13},   \eqref{e4.12} follows, as~claimed.~\end{proof}

\begin{remark}\label{r4.2z} \rm
\bit\item[(i)]It should be mentioned that, for $d=2$, Lemma \ref{l4.3a} fails~and this is the main reason our treatment of equation \eqref{e1.1} under as\-sump\-tions (j), (jj) is constrained to $d\ge3$.

\item[(ii)] We note that Remark \ref{r3.1prim} with the r\^ole of $H\1(\R^d)$ replaced by $\calh\1$ remains true in all its parts under condition (jj) above. We shall use this below without further notice.\eit\end{remark}
We denote by $j:\R\to\R$ the potential associated with $\psi$, that is, a continuous convex function on $\R$ such that $\pp j=\psi$, i.e.,
$$j(r)\le\zeta(r-\ov r)+j(\ov r),\ \ff\zeta\in\psi(r),\ r,\bar r\in\R.$$

\begin{definition}\label{d4.1} {\rm Let $x\in\calh\1$ and $p:=\max(2,2m)$.  An $\calh\1$-valued adapted process $X=X(t)$ is called strong solution  to \eqref{e1.1} if the \fwg\ conditions hold:
\begin{eqnarray}
&&X\mbox{ is $\calh\1$-valued continuous on $[0,T]$}, \pas,\label{e4.3}\\[2mm]
&&X\in L^p(\ooo\times(0,T)\times\R^d).\label{e4.4}
\end{eqnarray}There is $\eta\in L^{\frac pm}(\ooo\times(0,T)\times\R^d)$ such that
\begin{eqnarray}
&&\eta\in\psi(X),\ dt\otimes\mathbb{P}\otimes d\xi\ \mbox{ -- a.e. on }(0,T)\times\ooo\times\R^d\label{e4.5}
\end{eqnarray}and $\pas$
\begin{eqnarray}
&&X(t)=x+\Delta\dd\int^t_0\eta(s)ds+\sum^\9_{k=1}\mu_k\int^t_0X(s)e_kd\b_k(s)\label{e4.6}\\
&&\qquad\qquad\qquad\qquad\qquad\mbox{in }\mathcal{D}'(\R^d),\ \ t\in[0,T].\nonumber
\end{eqnarray}Here $\mathcal{D}'(\R^d)$ is the standard space of distributions on $\R^d$.}\end{definition}

Theorem \ref{t4.2} below is the main existence result for equation \eqref{e1.1}.

\begin{theorem}\label{t4.2} Assume that $d\ge3$ and that 
$$x\in L^{p}(\R^d)\cap L^2(\R^d)\cap\calh\1, \quad p:=\max(2,2m).$$ Then, under assumptions
 {\rm(j), (jj)}, there is a unique solution $X$ to \eqref{e1.1} such that
\begin{equation}\label{e4.7}
X\in L^2(\ooo;C([0,T];\calh^{-1})).
\end{equation}
Moreover, if $x\ge0,$ a.e. in $\R^d$, then $X\ge0$, a.e. on $(0,T)\times\R^d\times\ooo$.\end{theorem}

Theorem \ref{t4.2} is applicable to a large class of nonlinearities
$\psi:\R\to2^\R$ and, in particular, to
$$\psi(r)=\rho H(r)+\a r,\ \ff r\in\R, \psi(r) = \rho H(r-r_c) r,$$where $\rho>0$, $\a, r_c \ge0$, which models the dynamics
 of self-organized cri\-ti\-ca\-lity (see \cite{BDPR09}, \cite{5}, \cite{6}).
Here $H$ is the Heaviside function.

As mentioned earlier, Theorem \ref{t4.2} can be compared most closely to the main existence
 result of \cite{Ren}. But there are, however, a few notable differences as
 we explain below.  The function $\psi$ arising in \cite{Ren} is monotonically in\-crea\-sing, continuous and are assumed to satisfy a growth condition of the form $N(r)\le r\psi(r)\le C(N(r)+1)r,$ $\ff r\in\R$, where $N$ is a smooth and $\Delta_2$-regular Young function defining the Orlicz class $L_N$.  In contrast to this, here $\psi$ is any maximal monotone graph (multivalued) with arbitrary polynomial growth.

\bk\n{\bf Proof of Theorem \ref{t4.2}.}
Consider the approximating equation
\begin{equation}\label{e4.14}
\barr{l}
dX_\lbb-\Delta(\psi_\lbb(X_\lbb)+\lbb X_\lbb)dt=X_\lbb dW,\ t\in(0,T),\vsp
X_\lbb(0)=x,\earr\end{equation}where $\psi_\lbb=\frac1\lbb\,(1-(1+\lbb\psi)\1),\ \lbb>0$. We note that $\psi_\lbb=\pp j_\lbb$, where (see, e.g., \cite{barbu10})
$$j_\lbb(r)=\inf\left\{\frac{|r-\bar r|^2}{2\lbb}+j(\bar r);\ \bar r\in\R\right\},\ \ff r\in\R.$$

We have the following result.

\begin{lemma}\label{l4.3} Let $x\in\calh\1\cap L^p(\R^d)\cap L^2(\R^d)$, $p:=2m,$ $d\ge3$. Then \eqref{e4.14} has a unique solution
\begin{equation}\label{e4.15}
X_\lbb\in L^2(\ooo;C([0,T];\calh\1))\cap L^\9([0,T];L^p(\ooo\times\R^d)).\end{equation}Moreover, for all  $\lbb,\mu>0$, we have
\begin{eqnarray}\label{e4.16}
&\dd\E\sup_{0\le t\le T}\|(X_\lbb(t)-X_\mu(t)) \|^2_{-1}\le C(\lbb+\mu)\\[1mm]
&\dd\E|X_\lbb(t)|^p_p\le C|x|^p_p,\ \ff t\in[0,T],\label{e4.17}\\[1mm]
&\dd\E\int^T_0\int_{\R^d}|\psi_\lbb(X_\lbb)|^{\frac pm}dt\,d\xi\le C|x|^p_p,\ \ff\lbb>0,\label{e4.17a}\\[1mm]
&\dd\E\left[\sup_{0\le t\le T}\|X_\lbb(t)\|^2_{-1}\right]\le C\|x\|^2_{-1},\ \ff \lbb>0,\label{e4.17aa}
\end{eqnarray}where $C$ is independent of $\lbb,\mu$.\end{lemma}

\begin{proof}
We consider for each fixed $\lbb$ the equation (see \eqref{e3.6})
\begin{equation}\label{e4.18}
\barr{l}
dX^\nu_\lbb+(\nu-\Delta)(\psi_\lbb(X^\nu_\lbb)+\lbb X^\nu_\lbb)dt=X^\nu_\lbb dW\vsp
X^\nu_\lbb(0)=x,\earr\end{equation}where $\nu>0$. Let $x\in L^2(\R^d)\cap L^p(\R^d)\cap\calh\1$. By Claim 1 in the proof of Lemma \ref{l3.2},   \eqref{e4.18} has a unique solution $X^\nu_\lbb\in L^2(\Omega;L^\infty([0,T];L^2(\R^d)))\cap L^2(\Omega\times[0,T];H^1(\R^d))$ with continuous sample paths in $L^2(\R^d)$.

As seen in the proof of Theorem \ref{t3.1}, we have, for $\nu\to0$,
$$\barr{ll}
X^\nu_\lbb\to X_\lbb&\mbox{ strongly in }L^2(\ooo;C([0,T];H\1(\R^d)))\vsp
  &\mbox{ weak-star in }L^2(\Omega;L^\9([0,T];L^2(\R^d))),\vsp
  &\mbox{ and, by \eqref{e3.12prim}, along a subsequence also,}\earr$$
 where $X_\lbb$ is the solution to \eqref{e4.14}. It remains to be shown that $X_\lbb$ satisfies \eqref{e4.15}--\eqref{e4.17aa}. In order to explain the ideas, we apply first (formally) It\^o's formula to \eqref{e4.18} for the function $\varphi(x)=\frac1p\,|x|^p_p.$ We obtain
\begin{equation}\label{e4.19}
\barr{l}
\dd\frac1p\,\E|X^\nu_\lbb(t)|^p_p+
\E\int^t_0\int_{\R^d}(\nu-\Delta)(\psi_\lbb(X^\nu_\lbb)+
\lbb X^\nu_\lbb)|X^\nu_\lbb|^{p-2}X^\nu_\lbb ds\, d\xi\vsp
\qquad
\dd=\frac1p\,|x|^p_p+\frac{p-1}2\,\E
\int^t_0\int_{\R^d}\sum^\9_{k=1}\mu_k^2|X^\nu_\lbb e_k|^2|X^\nu_\lbb|^{p-2}dt\,d\xi.\earr\end{equation}
Taking into account that $X^\nu_\lbb,\psi_\lbb(X^\nu_\lbb)\in L^2(0,T;H^1(\R^d))$, $\pas$, by Claim 1 in the proof of Lemma \ref{l3.2}, we have

$$
%\barr{l}
\dd\int^t_0\int_{\R^d}(\nu-\Delta)(\psi_\lbb(X^\nu_\lbb)+\lbb X^\nu_\lbb)|X^\nu_\lbb|^{p-2}X^\nu_\lbb ds\,d\xi
%\vsp \qquad\quad
\ge\lbb(p-1)\dd\int^t_0\int_{\R^d}
|\nabla X^\nu_\lbb|^2
|X^\nu_\lbb|^{p-2}d\xi\,ds,%\earr
$$and
by \eqref{e4.2} we have
$$\E\int^t_0\int_{\R^d}\sum^\9_{k=1}\mu^2_k|X^\nu_\lbb e_k|^2|X^\nu_\lbb|^{p-2}ds\,d\xi
\le C_\9\E\int^t_0
\int_{\R^d}|X^\nu_\lbb|^pd\xi\, ds<\9.$$Then, we obtain by \eqref{e4.19} via Gronwall's lemma
\begin{equation}\label{e4.21a}
\E|X^\nu_\lbb(t)|^p_p\le C|x|^p_p,\ t\in(0,T),\end{equation}
and, by \eqref{e4.1},
\begin{equation}\label{e4.21aa}
\E\int^t_0\int_{\R^d}|\psi_\lbb(X^\nu_\lbb)|^{\frac pm}dt\,d\xi\le C|x|^p_p,\ t\in[0,T].\end{equation}
%%%%%%%%%%%%%%%%%%%%%
It should be said, however, that the above argument is formal, because the function $\vf$ is not of class $C^2$ on $L^2(\R^d)$ and we do not know a priori if the integral in the left side of \eqref{e4.19} makes sense, that is, whether $|X_\lbb^\nu|^{p-2}X^\nu_\lbb\in L^2(0,T;L^2(\ooo;H^1(\R^d)))$.  To make it rigorous,  we approximate $X^\nu_\lambda$ by a sequence $\{X_\lambda^{\nu,\varepsilon}\}$ of solutions to the equation
\begin{equation}\label{e4.19b}
\begin{array}{l}
dX^{\nu,\varepsilon}_\lambda+
A^{\nu,\vp}_\lbb(X^{\nu,\vp}_\lbb)
 dt=X^{\nu,\varepsilon}_\lambda dW,\vsp
X^{\nu,\varepsilon}_\lambda(0)=x.\end{array}\end{equation}

%%%%%%%%%%%%%

\n Here, $A^{\nu,\vp}_\lbb=\frac1\vp\,(I-(I+\vp A^\nu_\lbb)\1),\ \vp\in(0,1)$, is the Yosida approximation of the operator $A^{\nu}_\lbb x=(\nu-\Delta)(\psi_\lbb(x)+\lbb x)$, $\ff x\in D(A^\nu_\lbb)=H^1(\R^d)$. We set $J_\vp=(I+\vp A^\nu_\lbb)\1$ and note that $J_\vp$ is Lipschitz in $H=H\1(\R^d)$ as well as in all $L^q(\R^d)$ for $1<q<\9$. Moreover, we have
\begin{equation}\label{e4.21}
|J_\vp(x)|_q\le|x|_q,\ \ff x\in L^q(\R^d), \end{equation}
see \cite{BDR08}, Lemma 3.1. Since $A^{\nu,\vp}_\lbb$ is Lipschitz in $H$, equation \eqref{e3.1} has a unique adapted solution $X^{\nu,\vp}_\lbb\in L^2(\ooo;C([0,T];H)$ and by It\^o's formula we have
$$\frac12\,\E|X^{\nu,\vp}_\lbb(t)|^2_{-1}
\le\frac12\,|x|^2_{-1}+C_1
\sum^\9_{k=1}\mu^2_k\E\int^t_0|X^{\nu,\vp}_\lbb(s)e_k|^2_{-1}ds,$$which, by virtue of (jj), yields
\begin{equation}\label{4.22}\E|X^{\nu,\vp}_\lbb(t)|^2_{-1}\le C_2|x|^2_{-1},\ \ff \vp>0,\ x\in H.\end{equation}
Similarly, since $A^{\nu,\vp}_\lbb$ is Lipschitz in $L^2(\R^d)$ (see Lemma 4.6 below), we have also that $X^{\nu,\vp}_\lbb\in L^2(\ooo; C([0,T];L^2(\R^2)))$ and, again by It\^o's formula applied to the function $|X^{\nu,\vp}_\lbb(t)|^2_2$, we obtain 
$$\E|X^{\nu,\vp}_\lbb(t)|^2_2\le\frac12\,|x|^2_2+C_3
\sum^\9_{k=1}\mu^2_k\E\int^\9_0|X^{\nu,\vp}_\lbb(s)e_k|^2_2ds,$$
which yields, by virtue of (jj),
\begin{equation}\label{e4.23}
\E|X^{\nu,\vp}_\lbb(t)|^2_2\le C_4|x|^2_2,\ \ff t\in[0,T].\end{equation}

\n{\bf Claim 1.} {\it For $p\in[2,\9)$ and $x\in L^p(\R^d)$, we have that  $X^{\nu,\vp}_\lbb\in L^\9_W([0,T];$ $L^p(\ooo;L^p(\R^d))\cap L^2(\ooo;L^2(\R^d))),$  where here and below the subscript $W$ refers to $(\calf_t)$-adapted processes.}\bk

\begin{proof} For $R>0$, consider the set
$$\barr{lcl}
\calk_R&=&\{X\in L^\9_W([0,T];L^p(\ooo;L^p(\R^d))\cap L^2(\ooo;L^2(\R^d))),\vsp
&&\ \ e^{-p\alpha t}\E|X(t)|^p_p\le R^p,\ e^{-2\alpha t}\E|X(t)|^2_2\le R^2, t\in[0,T]\}.\earr$$Since, by \eqref{e4.19b}, $X^{\nu,\vp}_\lbb$ is a fixed point of the map
$$X\buildrel F\over\longrightarrow e^{-\frac t\vp}X+\frac1\vp\int^t_0 e^{-\frac{t-s}\vp}J_\vp(X(s))ds
+\int^t_0e^{-\frac{(t-s)}\vp}X(s)dW(s),$$obtained by iteration in $C_W([0,T];L^2(\ooo;H\cap L^2(\R^d)))$, it suffices to show that $F$ leaves the set $\calk_R$   invariant for $R>0$ large enough. By \eqref{e4.21}, we~have
\begin{equation}\label{e4.24}
\barr{l}
\dd\(e^{-p\a t}\E\left|e^{-\frac t\vp}x+\frac1\vp\int^t_0 e^{-\frac{t-s}\vp}J_\vp(X(s))ds\right|^p_p
\right)^{\frac1p}\vsp
\qquad\le\dd e^{-\(\frac1\vp+\a\)t}|x|_p
+e^{-\a t}\int^t_0\frac1\vp
e^{-\frac{(t-s)}\vp}
(\E|X(s)|^p_p)^{\frac1p}ds\vsp
\qquad\le\dd e^{-\(\frac1\vp+\a\)t}|x|_p+\frac {R}{1+\a\vp},\earr\end{equation}and, similarly, that
  \begin{equation}\label{e4.25}
%\barr{l}
\dd\(e^{-2\a t}\E\left|e^{-\frac t\vp}x+\frac1\vp\int^t_0 e^{-\frac{(t-s)}\vp}J_\vp(X(s))ds\right|^2_2\)
^{\frac12}
%\vsp
%\qquad\qquad
\le\dd e^{-\(\frac1\vp+\a\)t}|x|_2
+ \frac R{1+\a\vp}.
%\earr
  \end{equation}

Now, we set
$$Y(t)=\int^t_0 e^{-\frac{(t-s)}\vp}X(s)dW(s),\ t\ge0.$$We have
$$\barr{l}
dY+\dd\frac1\vp\ Y\,dt=X\,dW,\ \ t\ge0,\vsp
Y(0)=0.\earr$$Equivalently,
$$d(e^{\frac t\vp}Y(t))=e^{\frac t\vp}\,X(t)dW(t),\ t>0;\ Y(0)=0.$$
By Lemma  5.1  in \cite{13a}, it follows that $e^{\frac t\vp}Y$ is an $L^p(\R^d)$-valued $(\calf_t)$-adapted continuous process on $[0,\9)$ and
$$\E|e^{\frac t\vp}Y(t)|^p_p=
\frac12\,p(p-1)\sum^\9_{k=1}\mu^2_k
\E\int^t_0\int_{\R^d}|e^{\frac s\vp}Y(s)|^{p-2}|e^{\frac s\vp}X(s)e_k|^2ds.$$
This yields via Hypothesis (jj)
$$\E|e^{\frac t\vp}Y(t)|^p_p\le\frac12\,(p-1)\E\int^t_0
|e^{\frac s\vp}Y(s)|^p_pds+
C\E\int^t_0|e^{\frac s\vp}X(s)|^p_pds,\ \ff t\in[0,T],$$and, therefore,
$$\E|Y(t)|^p_p\le C_1e^{-\(\a+\frac 1\vp\)pt}\E\int^t_0|e^{\frac s\vp}X(s)|^p_pds
\le\frac{R^p e^{-p\a t}\vp C_1}{p(1+\vp\a )},\ \ff t\in[0,T].$$
Similarly, we get
$$e^{-2\a t}\E|Y(t)|^2_2\le\frac{R^2  \vp C_1}{2(1+\vp\a)}, \ \ff t\in[0,T].$$Then, by formulae \eqref{e4.24}, \eqref{e4.25}, we infer that, for $\a$ large enough and \mbox{$R>2(|x|_p+|x|_2),$} $F$ leaves  $\calk_R$ invariant, which proves Claim 1.\end{proof}

\bk\n{\bf Claim 2.} {\it We have, for all $p\in [2,\9)$  and $x\in L^p(\R^d)$, that there exists $C_p\in(0,\9)$ such that
\begin{equation}\label{e4.26}
\st{t\in[0,T]}{\mathrm
 ess\ sup}
\E|X_\lbb^{\nu,\vp}(t)|^p_p\le C_p\ \mbox{ for all }\vp,\lbb,\nu\in(0,1).\end{equation}}

\begin{proof} Again invoking Lemma 5.1 in \cite{13a}, we have by \eqref{e4.19b} that $X^{\nu,\vp}_\lbb$ satisfies
\begin{equation}\label{e4.27}
\barr{l}
\E|X^{\nu,\vp}_\lbb(t)|^p_p
=|x|^p_p-p\ \E\dd\int^t_0
\int_{\R^d}
A^{\nu,\vp}_\lbb(X^{\nu,\vp}_\lbb)
X^{\nu,\vp}_\lbb|X^{\nu,\vp}_\lbb|^{p-2}
d\xi\,ds\vsp
+p(p-1)\dd\sum^\9_{k=1}
\mu^2_k\E\int^t_0\int_{\R^d}
|X^{\nu,\vp}_\lbb|^{p-2}|X^{\nu,\vp}_\lbb e_k|^2d\xi\,ds.
\earr\end{equation}
On the other hand, $A^{\nu,\vp}_\lbb(X^{\nu,\vp}_\lbb)=\frac1\vp\,(X^{\nu,\vp}_\lbb-J_\vp(X^{\nu,\vp}_\lbb))$ and so we have
$$
%\barr{l}
\dd\int_{\R^d}
A^{\nu,\vp}_\lbb (X^{\nu,\vp}_\lbb)X^{\nu,\vp}_\lbb
|X^{\nu,\vp}_\lbb|^{p-2}d\xi
%\vsp\qquad
=\dd\frac1\vp\int_{\R^d}|X^{\nu,\vp}_\lbb|^pd\xi-
\dd\frac1\vp\int_{\R^r}J_\vp(X^{\nu,\vp}_\lbb)|X^{\nu,\vp}_\lbb|^{p-2}X^{\nu,\vp}_\lbb d\xi.
%\earr
$$
Recalling \eqref{e4.21}, we get, via the H\"older inequality,
$$\int_{\R^d}A^{\nu,\vp}_\lbb(X^{\nu,\vp}_\lbb)X^{\nu,\vp}_\lbb|X^{\nu,\vp}_\lbb|^{p-2}d\xi\ge0,$$and so, by \eqref{e4.27} and Hypothesis (jj), we obtain, via Gronwall's lemma, estimate \eqref{e4.26}, as claimed. \end{proof}

\mk\n{\bf Claim 3.} {\it We have, for $\vp\to0$,
$$X^{\nu,\vp}_\lbb\longrightarrow X^\nu_\lbb\mbox{ strongly in }L^\9_W([0,T];L^2(\ooo;H))$$
 and weakly$^*$ in $L^\9([0,T];L^p(\ooo;L^p(\R^d))\cap L^2(\ooo;L^2(\R^d))).$}\bk

\begin{proof} For simplicity, we write $X_\vp$ instead of $X^{\nu,\vp}_\lbb$ and $X$ instead of $X^\nu_\lbb$. Also, we set $\g(r)\equiv\psi_\lbb(r)+\lbb r$.

Subtracting equations \eqref{e4.19b} and \eqref{e4.18}, we get via It\^o's formula and because $A^{\nu,\vp}_\lbb$ is monotone on $H$
$$\barr{r}
\dd\frac12\,\E|X_\vp(t)-X(t)|^2_{-1,\nu}+
\E\dd\int^t_0\int_{\R^d}(\g(J_\vp(X))
-\g(X))(X_\vp-X)d\xi\,ds\vsp
\le C\E\dd\int^t_0|X_\vp(s)-X(s)|^2_{-1,\nu}ds,
\earr$$and hence, by Gronwall's lemma, we obtain
\begin{equation}\label{e4.26a}
\E|X_\vp(t)-X(t)|^2_{-1,\nu}\le C\E\int^T_0\int_{\R^d}
|\g(J_\vp(X))-\g(X)||X_\vp-X|d\xi\,ds.\end{equation}
On the other hand, it follows by \eqref{e4.21} that
$$\int_{\ooo\times[0,T]\times\R^d}|J_\vp(X)|^2\mathbb{P}(d\oo)dt\,d\xi
\le\dd\int_{\ooo\times[0,T]\times\R^d}|X|^2\mathbb{P}(d\oo)dt\,d\xi,$$while, for $\vp\to0$,
$$J_\vp(y)\longrightarrow y\mbox{ in }H\1,\ \ \ff y\in H\1,$$(because $A^{\nu,\vp}_\lbb$ is maximal monotone in $H\1(\R^d)$) and so, $J_\vp(X(t,\oo))\longrightarrow X(t,\oo)$ in $H\1(\R^d)$ for all $(t,\oo)\in(0,T)\times\ooo.$ Hence, as $\vp\to0$,
\begin{equation}\label{e4.26aa}
J_\vp(X)\longrightarrow X\mbox{ weakly in }L^2(\ooo\times[0,T]\times\R^d),\end{equation}and, according to the inequality above, this implies that, for $\vp\to0$,
$$|J_\vp(X)|_{L^2((0,T)\times\ooo\times\R^d)}\longrightarrow|X|_{L^2((0,T)\times\ooo\times\R^d)}.$$Hence, $J_\vp(X)\longrightarrow X\mbox{ strongly in }L^2(\ooo\times[0,T]\times\R^d)\mbox{ as }\vp\to0.$ Now, ta\-king into account that $\g$ is Lipschitz, we conclude by \eqref{e4.26a}, \eqref{e4.26aa} and by estimates \eqref{e4.23}, \eqref{e4.26} that Claim 3 is true.~\end{proof}

Now, we can complete the proof of Lemma \ref{l4.3}. Namely, letting first $\vp\to0$ and then $\nu\to\9$ in  \eqref{e4.26}, we get \eqref{e4.17} and hence \eqref{e4.17a} as desired.

%%%%%%%%%%%%%%%%%%%%

%To prove \eqref{e4.16} and \eqref{e4.17aa}, we apply the It\^o formula in \eqref{e4.18} to the function $$\phi(y)=\frac12\,|y|^2_{-1,\nu}:=\<(\nu I-\D)\1y,y\>_2.$$

Now, let us prove \eqref{e4.16} and \eqref{e4.17aa}. Arguing as in the proof of Theorem \ref{t3.1}, we obtain
\begin{equation}\label{e4.30z}
\barr{l}
\dd\frac12\,|X^\nu_\lbb(t)|^2_{-1,\nu}+
\dd\int^t_0\int_{\R^d}(\psi_\lbb(X^\nu_\lbb)+\lbb X^\nu_\lbb)X^\nu_\lbb d\xi\,ds\vsp
\qquad=\dd\frac12\,|x|^2_{-1,\nu}+\frac12\int^t_0\int_{\R^d}\sum^\9_{k=1}\mu^2_k|X^\nu_\lbb e_k|^2_{-1,\nu}d\xi\,ds\vsp
\qquad+\dd\int^t_0\<X^\nu_\lbb,X^\nu_\lbb dW\>_{-1,\nu}ds.\earr\end{equation}
Keeping in mind that, by \eqref{e4.12}, $|X^\nu_\lbb e_k|_{-1,\nu}\le C|X^\nu_\lbb|_{-1,\nu}(|e_k|_\9+|\na e_k|_d)$, where $C$ is independent of $\nu$, we obtain by
  the Burkholder-Davis-Gundy inequality for $p=1$ (cf. the proof of Theorem \ref{t3.1})
$$\E\sup_{t\in[0,T]}|X^\nu_\lbb(t)|^2_{-1,\nu}
+\lbb\E\int^T_0|X^\nu_\lbb|^2_2ds\le C|x|^2_{-1,\nu}.$$

Taking into account that
$$\lim_{\nu\to0}|y|_{-1,\nu}=\|y\|_{-1},\ \ff y\in\calh\1,$$
we obtain, as in Theorem \ref{t3.1} (see the part following \eqref{e319prime}),  that
\begin{equation}
\E\left[\sup_{t\in[0,T]}
\|X_\lbb(t)\|^2_{-1}\right]+
\lbb\E\int^T_0|X_\lbb(t)|^2_2dt\le C\|x\|^2_{-1},\ \ff\lbb>0,\end{equation}where $C$ is independent of $\lbb$. In particular, \eqref{e4.17aa} holds.

Completely similarly, one proves \eqref{e4.16}. Namely, we have
$$d(X^\nu_\lbb-X^\nu_\mu)+(\nu-\Delta)
(\psi_\lbb(X^\nu_\lbb)+\lbb X^\nu_\lbb-\psi_\mu(X^\nu_\mu)-\mu X^\nu_\mu)dt=(X^\nu_\lbb-X^\nu_\mu)dW$$and again   proceeding as in the proof of Theorem \ref{t3.1},   we obtain as above that

$$\barr{l}
\dd\frac12\,| X^\nu_\lbb(t)-X^\nu_\mu(t) |^2_{-1,\nu}\vsp
\qquad+
\dd\int^t_0\int_{\R^d}
(\psi_\lbb(X^\nu_\lbb)+\lbb X^\nu_\lbb-\psi_\mu(X^\nu_\mu)-\mu X^\nu_\mu)
(X^\nu_\lbb-X^\nu_\mu) d\xi\,ds\vsp
\qquad=\dd\frac12\int^t_0\int_{\R^d}
\sum^\9_{k=1}\mu^2_k|(X^\nu_\lbb-X^\nu_\mu) e_k|^2_{-1,\nu}ds\vsp
\qquad+\dd\int^t_0\<X^\nu_\lbb-X^\nu_\mu, (X^\nu_\lbb-X^\nu_\mu)dW \>_{-1,\nu},\ t\in[0,T].\earr$$
Then, applying once again  the Burkholder-Davis-Gundy inequality for \mbox{$p=1$,} and the fact that, by Hypothesis (j), $|\psi_\lbb(r)|\le C|r|^m,\ \ff \in\R$ with $C$ independent of $\lbb,$ we get, proceeding as in the proof of Theorem \ref{t3.1}, that
$$\E\left[\sup_{t\in[0,T]}| X^\nu_\lbb(t)-X^\nu_\mu(t) |^2_{-1}\right]\le C(\lbb+\mu),$$where $C$ is independent of $\nu,\lbb,\mu$. (For details, we refer to the proof of (3.10), (3.14) in \cite{BDPR09}). Letting $\nu\to0$  as in the previous case, we obtain \eqref{e4.16}, as claimed. This completes the proof of Lemma \ref{l4.3}.\end{proof}

Above we have used the  lemma below.

\begin{lemma}\label{l4.6} $A^{\nu,\vp}_\lbb$ is Lipschitz in $L^2(\R^d)$.\end{lemma}

\begin{proof} It suffices to check that $J_\vp$ is Lipschitz in $L^2(\R^d)$. We set $\g(r)=\psi_\lbb(r)+\lbb r$. We have, for $x,\bar x\in L^2(\R^d)$,
$$J_\vp(x)-J_\vp(\bar x)-\vp\Delta(\g(J_\vp(x))-\g(J_\vp(\bar x)))=x-\bar x.$$Multiplying by $\g(J_\vp(x))-\g(J_\vp(\bar x))$ in $L^2(\R^d)$, we get
$$\<J_\vp(x)-J_\vp(\bar x),\g(J_\vp(x))-
\g(J_\vp(\bar x))\>_2
\le|\g(J_\vp(x))-\g(J_\vp(\bar x))|_2|x-\bar x|_2.$$
Taking into account that $(\g(r)-\g(\bar r))(r-\bar r)\ge L|r-\bar r|$, $\ff r,\bar r\in\R$, and that $\g$ is Lipschitz, we get
$$|J_\vp(x)-J_\vp(\bar x)|_2\le C|x-\bar x|_2,$$as claimed.\end{proof}

\bk\n{\bf Proof of Theorem \ref{t4.2} (continued).} By \eqref{e4.16}-\eqref{e4.17aa}, it follows that there is a process $X\in L^\9([0,T];L^p(\ooo\times\R^d))$ such that, for $\lbb\to0$,
\begin{equation}\label{e4.22doi}
\barr{rcll}
X_\lbb&\to&X&\mbox{weak-star in $L^\9([0,T];L^p(\ooo\times\R^d))$}\vsp
\lbb X_\lbb&\to&0&\mbox{strongly in $L^2([0,T];L^2(\ooo\times\R^d))$}\vsp
\psi_\lbb(X_\lbb)&\to&\eta&\mbox{weakly in }L^{\frac pm}([0,T]\times\ooo\times\R^d)\vsp
X_\lbb&\to&X&\mbox{strongly in }L^2(\ooo;C([0,T];\calh\1)). \earr\end{equation}
It remains to be shown that $X$ is a solution to \eqref{e1.1} in the sense of Definition~\ref{d4.1}.

By \eqref{e4.14} and \eqref{e4.22doi}, we see that
\begin{equation}\label{e4.23doi}
\barr{l}
dX-\Delta\eta dt=XdW,\ t\in(0,T)\vsp
X(0)=x.\earr\end{equation}To prove that
$\eta\in \psi(X)$, a.e. in $\ooo\times(0,T)\times\R^d$, it suffices to show that, for each $\vf\in C^\9_0(\R^d)$, we have
\begin{equation}\label{e4.24doi}
\limsup_{\lbb\to0}\E\dd\int^T_0\int_{\R^d}\vf^2\psi_\lbb(X_\lbb)X_\lbb dt\,d\xi\le\E\dd\int^T_0\int_{\R^d}\vf^2\eta X \,d\xi\,dt.\end{equation}
Indeed, we have by convexity of $j_\lbb$
$$\barr{r}
\dd\E \int_0^T\!\! \int_{\R^d} \!\!\! \varphi^2 \psi_\lambda(X_\lambda)
(X_\lambda - Z) d\xi\,dt
 \ge \E\dd\int^T_0\int_{\R^d}
 \vf^2(j_\lbb(X_\lbb)-j_\lbb(Z))d\xi\,dt,
\vsp \forall Z \in L^p((0,T) \times \Omega \times \R^d),\earr$$
and so, by \eqref{e4.22doi} and  \eqref{e4.24doi}, we see that
$$\barr{r}
\dd\E \int_0^T \int_{\R^d}  \varphi^2(\eta  (X-Z)) dt d\xi   \ge \E\dd\int^T_0\!\!\!\int_{\R^d}
\!\!\vf^2(j(X)-j(Z))d\xi\,dt,
  \vsp \forall Z \in L^p((0,T) \times \Omega \times \R^d),\earr $$
  because, for $\lbb\to0$, $j_\lbb(Z)\to j(Z)$, and $j_\lbb(X_\lbb)\to j(X)$, a.e.  and thus, by Fatou's lemma
  $$\liminf_{\lbb\to0}
  \E\dd\int^T_0\int_{\R^d}\vf^2 j_\lbb(X_\lbb)d\xi\,dt\ge
  \E\dd\int^T_0\int_{\R^d}\vf^2 j(X)d\xi\,dt.$$
Now, we take  $\varphi \in C^\infty_0(\R^d)$ to be non-negative, such that
$\varphi = 1 $ on $B_N$ and $\varphi = 0 $, outside $B_{N+1}$
where for a given $N\in\N$, $B_N$ is the closed ball of $\R^d$
with radius $N$.
We get
\begin{equation}\label{e4.30prime}
\barr{r}
\dd\E \int_0^T \!\!\int_{B_{N+1}}  \!\!\! \varphi^2(\eta  (X-Z)) d\xi\,dt   \ge
\E\dd\int^T_0\!\!\int_{\R^d}\vf^2 (j(X)-j(Z)) d\xi\,dt,
  \vsp \forall Z \in L^p((0,T) \times \Omega \times \R^d).\earr
\end{equation}
This yields
\begin{equation}\label{e4.31z}
 \E\int^T_0\int_{B_{N+1}}\vf^2\eta(X-Z)d\xi\,dt
 \ge\E\int^T_0\int_{B_{N+1}}
 \vf^2\zeta(X-Z)d\xi\,dt,
\end{equation}for all $Z\in L^p((0,T)\times\ooo\times B_{N+1})$ and $\zeta\in L^{p'}((0,T)\times\ooo\times B_{N+1})$ such that $\zeta\in\psi(Z)$, a.e. in $(0,T)\times\ooo\times B_{N+1}$.

We denote by $\wt\psi:L^p((0,T)\times\ooo\times B_{N+1})\to L^{p'}((0,T)\times\ooo\times B_{N+1})$ the realization of the mapping $\psi$ in $L^p((0,T)\times\ooo\times B_{N+1})$, that is,
$$\wt\psi(Z)=\left\{\zeta\in L^{p'}((0,T)\times\ooo\times B_{N+1}),\ \zeta\in\psi(Z),\mbox{ a.e.}\right\}.$$
%%%%%%%
Since    $\frac{m}{p} \le
p'$ with $\frac1{p'}=1-\frac1p$, by virtue of assumption (j), $\wt\psi$ is maximal monotone in $L^p((0,T)\times\ooo\times B_{N+1})\times L^{p'}((0,T) \times \Omega \times B_{N+1})$, and so, the equation
\begin{equation}\label{e4.30zz}
 J(Z) + \wt\psi(Z)\ni J(X)+\eta,
\end{equation}
where $J(Z)=|Z|^{p-2}Z$, has a unique solution $(Z,\eta)$ (see, e.g., \cite{barbu10}, p. 31).

 If, in \eqref{e4.31z}, we take $Z$ the solution to \eqref{e4.30zz}, we obtain that
%%%%%%%%%
\begin{equation*}
\E \int_0^T \int_{B_{N+1}}  \ \varphi^2(J(X) - J(Z))(X-Z) dt d\xi   \le 0.
\end{equation*}
Then, choosing $\alpha = \frac{2}{p}$, yields
$$ \E \int_0^T \int_{B_{N+1}} \left(\vert \varphi^\alpha X\vert^{p-2}
\varphi^\alpha X - \vert \varphi^\alpha Z\vert^{p-2} \varphi^\alpha Z \right)
(\varphi^\alpha X - \varphi^\alpha Z) dt d\xi   \le 0.$$
Consequently, this gives
\begin{equation}\label{e4.30third}
 \E \int_0^T \int_{\R^d} (J(\varphi^\alpha X) - J(\varphi^\alpha Z))(\varphi^\alpha X - \varphi^\alpha Z) dt d\xi \le 0.
\end{equation}
 On the other hand, we have
 $$J(\vf^\a X)-J(\vf^\a Z)=(p-1)|\lbb\vf^\a X+(1-\lbb)\vf^\a Z|^{p-2}(X-Z),$$for some $\lbb=\lbb(X,Z)\in[0,1]$.  Substituting into \eqref{e4.30third} yields
 $$|\vf^\a(X-Z)|^2=0\mbox{\ \ a.e. in }(0,T)\times\ooo\times B_{N+1},$$Hence, $X=Z$ on  $(0,T)\times\ooo\times B_N.$

Coming back to \eqref{e4.30zz},  this gives $\eta \in \psi(X), \ dt dP d\xi,$   a.e., because $N$ is arbitrary.

 To prove \eqref{e4.24doi}, we use the It\^o 	 formula in \eqref{e4.18}  to $x\to\frac12\,\|\vf x\|^2_{-1}$ to get, as in \eqref{e4.30z},

\begin{equation*}
\barr{r}
\dd\frac12\,\E\| \vf X^\nu_\lbb(t)\|^2_{-1}+\E\int^t_0\<(-\Delta)\1(\nu-\Delta)
(\psi_\lbb(X^\nu_\lbb)+\lbb X^\nu_\lbb,\vf^2X^\nu_\lbb)\>ds\vsp
\le\dd\frac12\,\|\vf x\|^2_{-1}+\dd\frac12\,\E\int^t_0  \sum^\9_{k=1}\mu^2_k\|\vf X^\nu_\lbb e_k\|^2_{-1}ds.\earr\end{equation*}
Then, letting $\nu\to0$, we obtain
\begin{equation}\label{e4.25trei}
\barr{l}
\dd\frac12\,\E\| \vf X_\lbb(t)\|^2_{-1}+\E\int^t_0\<\psi_\lbb(X_\lbb)+\lbb X_\lbb,\vf^2X_\lambda  \>_2ds\vsp
\qquad\dd\le\frac12\,\|\vf x\|^2_{-1}+\dd\frac12\,
\E\int^t_0\sum^\9_{k=1}\mu^2_k\|\vf X_\lbb e_k\|^2_{-1}ds.\earr\end{equation}

On the other hand, by \eqref{e4.23doi} we get similarly
$$
%\barr{l}
\dd\frac12\,\E\|\vf X(t)\|^2_{-1}+\E\int^t_0\<\eta(s),\vf^2X \>_2 ds
%\vsp \qquad
=\dd\frac12\,\|\vf x\|^2_{-1}+\frac12\,\E\int^1_0\sum^\9_{k=1}\mu^2_k
\|\vf X e_k\|^2_{-1},\ t\in[0,T].
%\earr
$$Comparing with \eqref{e4.25trei}, we obtain \eqref{e4.24doi}, as claimed.

If $x\ge0$, a.e. in $\R^d$, it follows that $X\ge0$, a.e. in  in $\ooo\times(0,T)\times\R^d$. To prove this, one applies It\^o's formula in \eqref{e4.18} to the function $x\to |x^-|^2_2$ and get $(X^\nu_\lbb)^-=0$, a.e. in $\ooo\times(0,T)\times\R^d$. Then, for $\nu\to0$, we obtain the desired result.  This completes the existence proof for  $x\in L^2(\R^d)\cap L^p(\R^d)\cap\calh\1$.\bk

\n{\bf Uniqueness.} If $X_1,X_2$ are two solutions, we have
$$\barr{l}
d(X_1-X_2)-\Delta(\eta_1-\eta_2)dt=(X_1-X_2)dW,\ t\in(0,T),\vsp
(X_1-X_2)(0)=0,\earr$$where $\eta_i\in\psi(X_i),\ i=1,2,$ a.e. in $\ooo\times(0,T)\times\R^d.$

Applying again, as above (that is, via the approximating device)  It\^o's formula in $\calh\1$ to $\frac12\,\|\vf(X_1-X_2)\|^2_{-1}$, where $\vf\in C^\9_0(\R^d)$, we get that
$$\barr{l}
\dd\frac12\,d\|\vf(X_1-X_2) \|^2_{-1}-
\<\Delta(\eta_1-\eta_2),\vf (X_1-X_2)\>_{-1}\vsp
=\dd\frac12\dd\sum^\9_{k=1}\mu^2_k\|\vf
(X_1-X_2) e_k\|^2_{-1}dt
+\<(X_1-X_2) ,\vf(X_1-X_2)dW\>_{-1}=0.\earr$$
Note that, since $\eta_1-\eta_2\in L^{\frac pm}(\ooo\times(0,T)\times\R^d)$, we have
$$
%\barr{l}
\dd-\E\int^T_0\<\Delta(\eta_1-\eta_2), \vf(X_1-X_2)\>_{-1}dt
%\vsp \qquad
=
\E\dd\int^T_0\int_{\R^d}(\eta_1-\eta_2), \vf(X_1-X_2)dt\,d\xi \ge0,
%\earr
$$and, therefore,
$$\E\|\vf(X_1(t)-X_2(t)) \|^2_{-1}\le C\int^t_0\E\|\vf(X_1-X_2) \|^2_{-1}ds,\ \ff t\in[0,T],$$and, since $\vf$ was arbitrary in $C^\9_0(\R^d)$, we get
 $X_1\equiv X_2$, as claimed.

\begin{remark}\label{r4.1z} {\rm The self-organized criticality model \eqref{1.3}, that is, $\psi(r)\equiv H(r)=\mbox{Heaviside function}$, which is not covered by Theorem \ref{t4.2} for $1\le d\le 2$, can, however, be treated in the special case
$$W(t)=\sum^N_{j=1}\mu_j\b_j(t),\ \mu_j\in\R,$$(i.e., spatially independent noise) via the rescaling transformation $X=e^WY$, which reduces it to the random parabolic equation
$$\frac \partial {\partial t}\,Y-e^{-W}\Delta\psi(Y)+\frac12\sum^N_{j=1}\mu^2_jY=0.$$By approximating $W$ by a smooth $W_\vp\in C^1([0,T];\R)$ and letting $\vp\to0$, after some calculation one concludes that the latter equation has a unique strong solution $Y$. We omit the details, but refer to \cite{3a} for a related treatment.}
\end{remark}

\section{The finite time extinction}
\setcounter{equation}{0}

Assume here that $\psi$ satisfies condition (j) of
the beginning of Section \ref{s4}
 and that $W$ is of the form~(jj). Moreover, one assumes that
\begin{equation}\label{e5.1}
\psi(r)r\ge\rho|r|^{m+1},\ \ff r\in\R,\end{equation}
where $m$ is as in Hypothesis (j).

\begin{theorem}\label{t5.1} Let $d\ge3$ and $m=\frac{d-2}{d+2}$. Let $x\in L^{m+1}(\R^d)\cap
L^2(\R^d)\cap\calh\1$ and let $X=X(t);\ t\in[0,T]$, be the
solution to \eqref{e1.1} given by Theorem {\rm\ref{t4.2}}. We set
\begin{equation}\label{5.1a}
\tau=\inf\{t\ge0;\ \|X(t,\cdot)\|_{-1}=0\}.\end{equation}
 Then, for every $t > 0$,
\begin{equation}
X(t)=0,\ \ff t\ge\tau,\ \label{e5.2}\\[2mm]
\end{equation}
and
\begin{equation}\label{e5.3}
\mathbb{P}[\tau\le t]\ge 1 - \|x\|^{1-m}_{-1} \frac{C^\ast}{\rho  \gamma^{m+1} (1- e^{-C^\ast(1-m) t})}.
\end{equation}
where $\g\1=\sup\{\|u\|_{-1}|u|^{-1}_{m+1};\ u\in
L^{m+1}\}$
and $C^*>0$ is independent of the initial condition $x$.
\end{theorem}

\begin{proof}  We follow the arguments of \cite{5}. The basic inequality is
\begin{equation}\label{e5.4}
\barr{l}
\|X(t)\|^{1-m}_{-1}+\rho(1-m)\g^{m+1}\dd\int^t_r\one_{[\|X(s)\|_{-1}>0]}ds\vsp
\qquad\le\|X(r)\|^{1-m}_{-1}+C^*(1-m)\dd\int^t_r\|X(s)\|^{1-m}_{-1}ds\vsp
\qquad+(1-m)\dd\int^t_r\<\|X(s)\|^{(m+1)}_{-1}X(s),X(s)dW(s)\>_{-1},\vsp
\hfill \pas,\ 0<r<t<\9,\earr\end{equation}where $C^*$ is a suitable constant.
(We note that, by virtue of \eqref{e2.8}, \mbox{$\g\1<\9$.}) To get \eqref{e5.4}, we apply the It\^o formula in \eqref{e4.14} to the semimartingale $\|X_\lbb(t)\|^2_{-1}$ and to the function $\vf_\vp(r)=(r+\vp^2)^{\frac{1-m}2},$ $r>-\vp^2$, where $X_\lbb$ is the solution to \eqref{e4.14}.

We have

$$\barr{l}
d\vf_\vp(\| X_\lbb(t)\|^2_{-1})\vspace*{2mm}\\
 +(1-m)(\|  X_\lbb(t)\|^2_{-1}+\vp^2)^{-\frac{m+1}2}\<  X_\lbb(t),\psi_\lbb(X_\lbb(t))+\lbb X_\lbb(t)\>_2dt\vspace*{2mm}\\
 =\dd\frac12\sum^\9_{k=1}\mu^2_k
\left[\frac{(1{-}m)\| X_\lbb(t)e_k
\|^2_{-1}}
{(\|  X_\lbb(t)\|^2_{-1}{+}\vp^2)^{\frac{m+1}2}}
-(1{-}m^2)\frac{\| X_\lbb(t)e_k\|^2_{-1}\| X_\lbb(t)\|^2_{-1}}{(\|  X_\lbb(t)\|^2_{-1}{+}
\vp^2)^{\frac{m+1}2}}\right]dt\vspace*{3mm}\\
 +2\<\vf'_\vp
(\|X_\lbb(t)\|^2_{-1})X_\lbb(t),X_\lbb(t)dW(t)\>.\earr$$
This yields\vspace*{-3mm}
$$\barr{l}
\vf_\vp(\| X_\lbb(t)\|^2_{-1})+\rho(1-m)\dd\int^t_r(\| X_\lbb(s)\|^2_{-1}+\vp^2)^{-\frac{m+1}2}
\dd\int_{\R^d} |X_\lbb|^{m+1}ds\,d\xi\vsp
\qquad\le\vf_\vp(\|  X_\lbb(r)\|^2_{-1} )
+C^*\dd\int^t_r\|X_\lbb(s)\|^2_{-1}(\| X_\lbb(s)\|^2_{-1}+\vp^2)^{-\frac{1+m}2}ds\vsp
\qquad+2\dd\int^t_r\<\vf'_\vp(\| X_\lbb(s)\|^2_{-1}) X_\lbb(s),X_\lbb(s)dW(s)\>_{-1}.\vspace*{-2mm}
\earr$$
Now, letting $\lbb\to0$, we obtain that $  X $ satisfies the estimate\vspace*{-2mm}
\begin{equation}\label{e5.5}
\barr{l}
\vf_\vp(\|  X(t)\|^2_{-1})+\rho(1-m)\dd\int^t_r\!\!\(\| X(s)\|^2_{-1}
+\vp^2)^{-\frac{m+1}2}\!\!\int_{\R^d} \! |X(s,\xi)|^{m+1}d\xi\!\)\!ds\vsp
\le\vf_\vp(\|X(t)\|^2_{-1})+C^*\dd\int^t_r\|  X(s)\|^2_{-1}(\| X(s)\|^2_{-1}+\vp^2)^{-\frac{m+1}2}ds \vsp
+2\dd\int^t_r\<\vf'_\vp(\|  X(s)\|^2_{-1}) X(s),X(s)dW(s)\>_{-1}.
\earr\hspace*{-9mm}\end{equation} Here, we have used the fact that, by Lemma \ref{l4.3}, for $\lbb\to0$,
$$  X_\lbb\to   X\mbox{ in }\calh\1,  $$and, by \eqref{e4.22doi} it follows, via Fatou's lemma,
$$\liminf_{\lbb\to0}\int_{\R^d}  |X_\lbb|^{m+1}d\xi\ge\int_{\R^d} |X|^{m+1}d\xi,$$
and
$$\barr{l}
(\|X(t)\|^2_{-1}+\vp^2)^{\frac{1-m}2}+\rho(1-m)\g^{m+1}\dd\int^t_r
(\|X(s)\|^2_{-1}+\vp^2)^{-\frac{m+1}2}\|X(s)\|^{m+1}_{-1}ds\vsp
\qquad\le(\|X(r)|^2_{-1}+\vp^2)^{\frac{1-m}2}+
C^*\dd\int^t_r\|X(s)\|^2_{-1}(\|X(s)\|^2_{-1}+\vp^2)^{-\frac{m+1}2}ds\vsp
\qquad+2\dd\int^t_r\<\vf'_\vp
(\|X(s)\|^2_{-1}X(s)),X(s)dW(s)\>_{-1},\ 0\le r\le t<\9,\earr$$
because, by \eqref{e2.7},  $\|x\|_{-1}\le \g\1|x|_{m+1},\ \ff x\in L^{m+1}(\R^d)$. Letting $\vp\to0$, we get \eqref{e5.4}, as claimed.

Now, we conclude the proof as in \cite{5}. Namely, by \eqref{e5.4}, it follows that
$$\barr{l}
e^{-C^*(1-m)t}\|X(t)\|^{1-m}_{-1}+\rho(1-m)\g^{m+1}\dd\int^t_r
e^{-C^*(1-m)s}\one_{[\|X_s\|_{-1}>0]}ds\vsp
\qquad\le e^{-C^*(1-m)r}\|X(r)\|^{1-m}_{-1}\vsp
\qquad\dd+(1-m)\dd\int^t_re^{C^*(1-m)s}\<\|X(s)\|^{-(m+1)}_{-1}X(s),X(s)dW(s)\>_{-1}\earr$$
and, therefore, $t\to e^{-C^*(1-m)t}\|X(t)\|^{1-m}_{-1}$ is an $\{\calf_t\}$
 supermartingale. Hence, $\|X(t)\|_{-1}=0$ for $t\ge\tau$,
because of Proposition 3.4, Chap. 2 of \cite{revuz}.
 Moreover,  taking expectation for $r = 0$, we get
$$ e^{-C^*(1-m)t}\E\|X(t)\|^{1-m}_{-1}+\rho(1-m)\g^{m+1}
\int^t_0e^{-C^*(1-m)s}\mathbb{P}(\tau>s)ds\le\|x\|^{1-m}_{-1}.$$
This implies that
$$
 \mathbb{P}(\tau > t) \frac{1 - e^{-C^*(1-m)t}}{C^*(1-m)} \le
\int^t_0e^{-C^*(1-m)s}\mathbb{P}(\tau>s)ds \le
\frac{\|x\|^{1-m}_{-1}}{\rho (1-m) \gamma^{m+1}},
$$
and so \eqref{e5.3} follows. This completes the proof.\end{proof}

\begin{corollary} \label{c5.2}
 Let $ x \in {\cal H}^{-1}\cap L^{m+1}(\R^d)\cap L^2(\R^d)$ be such that
$\Vert x \Vert_{-1} < \frac{\rho \gamma^{m+1}}{C^\ast}$.
Let $\tau$ be
the stopping time defined in \eqref{5.1a}.
Then $\mathbb{P}(\tau<\9)>0$.
In other words, there is extinction in finite time with positive probability.
\end{corollary}

\begin{remark}\label{r5.1} {\rm In the case of bounded domain, Theorem \ref{t5.1} remains true for $m\in\left[\frac{d-2}{d+2},1\right)$ (see \cite{5}). One might suspect that also in this case the extinction property \eqref{e5.3} holds for a larger class of exponents $m$. However, the analysis carried out in \cite{vazquezDecay} for deterministic fast diffusion equations in $\R^d$ shows that the extinction property is dependent not only on the exponent $m$, but also on the space $L^p(\R^d)$, where the solution exists (the so called extinction space).}\end{remark}

\begin{remark}\label{r5.2}\rm\
\begin{itemize}\item
The analysis in this section holds, in particular, if all
the coef\-fi\-cients $\mu_k$ do vanish, i.e., in the deterministic
framework. In that case, Theorem \ref{t5.1} implies the existence
of a deterministic time $\tau > 0$ so that
$$ t \ge \tau \Rightarrow \Vert X(t) \Vert_{-1} = 0,$$
and so $X(t) = 0 $, for all $t \ge \tau$.
\item Let us set, for instance, $\psi(u)=u^m,\ d\ge3,\ m=\frac{d-2}{d+2}\cdot$. Observe that  $(L^1\cap L^\9)(\R^d)\subset (L^{m+1}\cap L^2)(\R^d)\cap \calh^{-1}. $ Consider, for instance, as initial condition $x\in (L^1\cap L^\9)(\R^d).$
    \item By the Benilan-Crandall approach, see, e.g., Theorem 1 of \cite{BeC81}, there is a solution $u:[0,T]\times\R^d\to\R$, of
        \begin{equation}\label{e5.7}
        \barr{l}
        dX-\Delta\psi(X)dt=0\ \mbox{ in }(0,T)\times\R^d,\vsp
        X(0)=x\ \mbox{ on }\R^d,\earr\end{equation}
        in the sense of distributions. $u$ belongs to $(L^1\cap L^\9)((0,T)\times\R^d)$ and also the $\eta_u=\psi(u)$. In this case, $u$ fulfills mass conservation.
        \item By use of Theorem \ref{t4.2}, there is another solution $v:[0,T]\times\R^d\to\R^d$ in the sense of distributions, such that $v\in L^p((0,T)\to\R^d)$, with $p=\max(1,2m)$. Also, $\eta_v=\psi(v)\in L^{\frac pm}((0,T)\times\R^d)$. By Theorem \ref{t5.1}, if $x$ is small enough, there will be extinction, and so, $v$ does not fulfill any mass conservation.

\item In particular, there is no uniqueness for \eqref{e5.7} in the sense of distributions. Remark that according to \cite{BrC79}, uniqueness is guaranteed in the class $(L^1\cap L^\9)((0,T)\times\R^d).$
    \end{itemize}

\end{remark}

\n{\bf Acknowledgements.}
Financial support through the SFB 701 at Bielefeld University and
NSF-Grant 0606615
is gratefully acknowledged. The   third named author was partially supported
by the ANR Project MASTERIE 2010 BLAN 0121 01.

\bigskip
%{\bf References}
\bibliographystyle{elsarticle-harv}
\bibliography{BRR_Bibliography}

%% Authors are advised to submit their bibtex database files. They are
%% requested to list a bibtex style file in the manuscript if they do
%% not want to use elsarticle-harv.bst.

%% References without bibTeX database:

% \begin{thebibliography}{00}

%% \bibitem must have one of the following forms:
%%   \bibitem[Jones et al.(1990)]{key}...
%%   \bibitem[Jones et al.(1990)Jones, Baker, and Williams]{key}...
%%   \bibitem[Jones et al., 1990]{key}...
%%   \bibitem[\protect\citeauthoryear{Jones, Baker, and Williams}{Jones
%%       et al.}{1990}]{key}...
%%   \bibitem[\protect\citeauthoryear{Jones et al.}{1990}]{key}...
%%   \bibitem[\protect\astroncite{Jones et al.}{1990}]{key}...
%%   \bibitem[\protect\citename{Jones et al., }1990]{key}...
%%   \harvarditem[Jones et al.]{Jones, Baker, and Williams}{1990}{key}...
%%

% \bibitem[ ()]{}

% \end{thebibliography}

\end{document}